\newif\iffinal
\finalfalse	

\documentclass[letterpaper,11pt,reqno]{amsart} 

\RequirePackage[utf8]{inputenc}
\usepackage[portrait,margin=3cm]{geometry}
\usepackage{color}              
\usepackage[usenames,dvipsnames]{xcolor}

\usepackage{mathrsfs,soul}
\usepackage{hyperref}
\usepackage{amssymb,amsthm,amsmath,amsfonts,amsbsy,latexsym,dsfont}
\usepackage[textsize=small]{todonotes}       
\usepackage{graphicx}
\usepackage[numeric,initials,nobysame]{amsrefs}
\usepackage{upref,setspace}

\definecolor{BrickRed}{rgb}{0.65,0.08,0}

\usepackage{enumerate}

\numberwithin{equation}{section}
\numberwithin{figure}{section}
\numberwithin{table}{section}

\newtheorem{Lemma}{Lemma}[section]
\newtheorem{Proposition}{Proposition}[section]

\newtheorem{Theorem}{Theorem}[section]

\newtheorem{Condition}{Condition}[section]

\newtheorem{Remark}{Remark}[section]

\newtheorem{Corollary}{Corollary}[section]

\newcommand{\half}{{\frac{1}{2}}}

\newcommand{\lan}{\langle}
\newcommand{\ran}{\rangle}

\newcommand{\Rd}{{\Rmb^d}}

\newcommand{\Cmb}{{\mathbb{C}}}

\newcommand{\Emb}{{\mathbb{E}}}

\newcommand{\Nmb}{{\mathbb{N}}}

\newcommand{\Pmb}{{\mathbb{P}}}

\newcommand{\Rmb}{{\mathbb{R}}}
\newcommand{\Smb}{{\mathbb{S}}}

\newcommand{\Bmc}{{\mathcal{B}}}

\newcommand{\Gmc}{{\mathcal{G}}}

\newcommand{\Lmc}{{\mathcal{L}}}

\newcommand{\Pmc}{{\mathcal{P}}}

\newcommand{\Rmc}{{\mathcal{R}}}

\newcommand{\Tmc}{{\mathcal{T}}}

\newcommand{\one}{{\boldsymbol{1}}}

\newcommand{\muhat}{{\hat{\mu}}}

\newcommand{\Ctil}{{\tilde{C}}}

\newcommand{\ptil}{{\tilde{p}}}

\newcommand{\Rmctil}{{\widetilde{\Rmc}}}

\newcommand{\Tmctil}{{\widetilde{\Tmc}}}

\numberwithin{equation}{section}

\newcommand{\emppre}{{\nu}}

\newcommand{\emplimit}{{\bar{\nu}}}
\newcommand{\lawlimit}{{\bar{\mu}}}
\newcommand{\lawlimitmean}{{\hat{\mu}}}
\newcommand{\Xlimit}{{\bar{X}}}

\begin{document}

	\title[Graphon particle system: Concentration bounds]{Graphon particle system: Uniform-in-time concentration bounds}

	\makeatletter
	\@namedef{subjclassname@2020}{\textup{2020} Mathematics Subject Classification}
	\makeatother

	\subjclass[2020]{%
	05C80
	60J60  	
	60K35%
	}
	\keywords{%
	graphons,
	graphon particle systems,
	mean field interaction,
	heterogeneous interaction,
	networks,
	long time behavior,
	exponential concentration bounds,
	transport inequalities%
	}
	 
	\author[Bayraktar]{Erhan Bayraktar}
	\address{Department of Mathematics, University of Michigan, 530 Church Street, Ann Arbor, MI 48109}
		\thanks{E. Bayraktar is supported in part by the National Science Foundation under DMS-2106556, and in part by the Susan M. Smith Professorship.}  
	\author[Wu]{Ruoyu Wu$^*$}
	\address{Department of Mathematics, Iowa State University, 411 Morrill Road, Ames, IA 50011} 
	\email{erhan@umich.edu, ruoyu@iastate.edu}
			 
	\begin{abstract}	
	In this paper, we consider graphon particle systems with heterogeneous mean-field type interactions and the associated finite particle approximations. Under suitable growth (resp.\ convexity) assumptions, we obtain uniform-in-time concentration estimates, over finite (resp.\ infinite) time horizon, for the Wasserstein distance between the empirical measure and its limit, extending the work of Bolley--Guillin--Villani \cite{BolleyGuillinVillani2007quantitative}.
	\end{abstract}	
	
	\maketitle

	\tableofcontents

\section{Introduction}

In this work we study the uniform-in-time exponential concentration bounds related to the graphon particle system and its finite particle approximations.
This is a continuation of our work \cite{BayraktarChakrabortyWu2020graphon,BayraktarWu2020stationarity} on the large population and long time behavior for these systems.
The interaction in the graphon particle system is of mean-field type and characterized by a (directed) graphon $G$, which is a measurable function from $[0,1] \times [0,1]$ to $[0,1]$ (see e.g.\ \cite{Lovasz2012large} for the theory of graphons).
More precisely, denoting by $\Xlimit_u(t)$ the state of the particle indexed by $u \in [0,1]$ at time $t \ge 0$,
\begin{equation}
	\Xlimit_u(t) = \Xlimit_u(0) + \int_0^t \left( f(\Xlimit_u(s)) + \int_0^1 \int_{\Rmb^d} b(\Xlimit_u(s),x)G(u,v) \,\lawlimit_{v,s}(dx)\,dv\right)ds + B_u(t), \label{eq:system}
\end{equation}
where $\lawlimit_{v,s}$ is the probability distribution of the $\Rd$-valued random variable $\Xlimit_v(s)$ for each $v \in [0,1]$ and $s \ge 0$, $f$ are $b$ are suitable functions, $\{B_u : u \in [0,1]\}$ are $d$-dimensional standard Brownian motions, and $\{\Xlimit_u(0),B_u : u \in [0,1]\}$ are mutually independent.
Due to the general form of $G$, the system \eqref{eq:system} consists of uncountably many heterogeneous particles $\Xlimit_u$, the evolution of whose probability distributions $\lawlimit_{u,t}$ are fully coupled.
Indeed, $\{(\Xlimit_u(0),B_u) : u \in [0,1]\}$ are not necessarily identically distributed and may not be measurable in $u \in [0,1]$.
Although the system \eqref{eq:system} of nonlinear McKean--Vlasov type processes $\Xlimit_u(t)$ is well-posed \cite{BayraktarChakrabortyWu2020graphon}, we will be only interested in the concentration of finite particle empirical measures around the probability laws $\lawlimit_{u,t}$, and hence one may alternatively treat \eqref{eq:system} as an informal description and work with equivalent formulations that do not involve uncountable $\{B_u : u \in [0,1]\}$; see Remark \ref{rmk:formulation} on more of this point. 
Still, in general, $\lawlimit_{u,t}$ and the associated stationary measure, provided existence under suitable assumptions, may not be tractable (even in the case of linear $f$ and $b$; see e.g.\ \cite[Example 3.1]{BayraktarWu2020stationarity}).
However, the system \eqref{eq:system} arises naturally as the limit (see e.g.\ \cite{BayraktarChakrabortyWu2020graphon,BayraktarWu2020stationarity}) of the associated finite particle system with heterogeneous interactions given by
\begin{equation}
	X_i^n(t) = \Xlimit_{\frac{i}{n}}(0) + \int_0^t \left( f(X_i^n(s)) + \frac{1}{n} \sum_{j=1}^n \xi_{ij}^n b(X_i^n(s),X_j^n(s)) \right) ds + B_{\frac{i}{n}}(t)
	\label{eq:system-n}
\end{equation}
for $i \in \{1,\dotsc,n\}$ and $t \ge 0$,
where $\{\xi_{ij}^n: 1 \le i \le j \le n\}$ is a collection of independent $[0,1]$-valued random variables sampled from the graphon $G$.
The study of uniform-in-time exponential concentration bounds will provide useful quantitative non-asymptotic estimates of the fluctuations for the convergence of the empirical measures of \eqref{eq:system-n} to the probability laws of \eqref{eq:system}.

When $G \equiv 1$ and $\xi_{ij}^n \equiv 1$, the above systems \eqref{eq:system} and \eqref{eq:system-n} reduce to the classic McKean--Vlasov processes and associated weakly interacting diffusions.
The study of this \textit{homogeneous} setup dates back to works of Boltzmann, Vlasov, McKean and others (see \cite{Sznitman1991,McKean1967propagation,Kolokoltsov2010} and references therein).  
Besides large population asymptotics such as law of large numbers (LLN) on the finite time horizon, there have been an extensive collection of results on concentration estimates and uniform-in-time LLN (see e.g.\ \cite{BudhirajaFan2017,Veretennikov2006ergodic,BolleyGuillinVillani2007quantitative,DelarueLackerRamanan2020from,BolleyGuillinMalrieu2010trend} and references therein). 

In recent ten years, there has been a growing interest in the mean-field \textit{inhomogeneous} particle system, where $G \equiv 1$ (or block-wise constant) and the interaction between particles is governed by their own types and/or random graphs (see e.g.\ \cite{BayraktarWu2019mean,BarreDobsonOttobreZatorska2021fast,BhamidiBudhirajaWu2019weakly,Delattre2016,CoppiniDietertGiacomin2019law,BudhirajaMukherjeeWu2019supermarket,Delarue2017mean,LackerSoret2020case}).
The inhomogeneity arises from random interactions $\xi_{ij}^n$ while the limiting system is still homogeneous, as opposed to uncountable heterogeneous processes in graphon particle systems \eqref{eq:system} with general graphon $G$.

The study of graphon particle systems and associated finite particle models with mean-field \textit{heterogeneous} interactions emerged recently (\cite{BayraktarChakrabortyWu2020graphon,BayraktarWu2020stationarity,BetCoppiniNardi2020weakly,Lucon2020quenched,OliveiraReis2019interacting,Coppini2019long,Coppini2021note}). 
There is also a growing number of applications of graphons in game theory; see e.g.\ \cite{Carmona2019stochastic,PariseOzdaglar2019graphon,CainesHuang2018graphon,GaoTchuendomCaines2020linear,VasalMishraVishwanath2020sequential,CainesHuang2020graphon,BayraktarWuZhang2022propagation} for the study of graphon mean field games in static and dynamic settings.
These works focus on large population convergence over finite time horizon, except for \cite{Coppini2019long,BayraktarWu2020stationarity} on uniform-in-time convergence,
where the limit is homogeneous in \cite{Coppini2019long} while heterogeneous in \cite{BayraktarWu2020stationarity}.
Deterministic dynamical systems on graphons over finite time horizon are also studied (see \cite{Medvedev2014nonlinear1,Medvedev2014nonlinear2,Medvedev2018continuum,DupuisMedvedev2020large} and references therein).

The analysis of the continuum limit \eqref{eq:system} is motivated by physical or biological applications such as in \cite{Coppini2019long,Coppini2021note,CoppiniDietertGiacomin2019law,Lucon2020quenched,Medvedev2014nonlinear2,Medvedev2014nonlinear1,Medvedev2018continuum,OliveiraReis2019interacting}, by numerical questions as in \cite{BolleyGuillinVillani2007quantitative}, or by game theory applications such as in \cite{Carmona2019stochastic,PariseOzdaglar2019graphon,CainesHuang2018graphon,GaoTchuendomCaines2020linear,VasalMishraVishwanath2020sequential,CainesHuang2020graphon}.
In this work we are interested in the approximation errors of \eqref{eq:system} by \eqref{eq:system-n} by obtaining non-asymptotic exponential concentration results for quantities like 
$$\Pmb(\sup_{0 \le t \le T} W_1(\emppre^n_t, \lawlimitmean_t) > \varepsilon) \quad \mbox{ and } \quad \sup_{t \ge 0} \Pmb(W_1(\emppre^n_t, \lawlimitmean_t) > \varepsilon)$$ 
for $\varepsilon>0$, in Theorem \ref{thm:finite} (resp.\ Theorem \ref{thm:infinite}) under suitable growth (resp.\ convexity) condition.
Here $W_1$ is the Wasserstein-1 distance, $\emppre^n_t$ is the empirical measure of $\{X_i^n(t):i=1,\dotsc,n\}$, and $\lawlimitmean_t$ is the averaged law of $\{\Xlimit_u(t):u \in [0,1]\}$.
As an application, we also obtain in Corollary \ref{cor:invariant} an exponential concentration bound for numerical reconstruction of the invariant measure of $\{\Xlimit_u(t):u \in [0,1]\}$.
Concentration estimates are shown in  \cite{BudhirajaFan2017,BolleyGuillinVillani2007quantitative,BolleyGuillinMalrieu2010trend} for mean-field (type) systems over infinite time horizon and in \cite{DelarueLackerRamanan2020from} for mean-field games over finite time horizon.
Theorems \ref{thm:finite} and \ref{thm:infinite} provide for the first time quantitative non-asymptotic exponential error bounds for graphon particle systems.
To the best of our knowledge, there are no results on concentration bounds in either inhomogeneous or heterogeneous regime, over either finite or infinite time horizon. 
Our results are consistent with those in \cite{BolleyGuillinVillani2007quantitative} when the system is mean-field (namely $G \equiv 1$ and $\xi_{ij}^n \equiv 1$), and actually improve pre-exponential estimates in the infinite time horizon (see Remark \ref{rmk:improvement}).
Although only the dissipative regime is considered for the infinite time horizon, the results here suggest that it is indeed possible to obtain concentration bounds with asymptotic independence instead of the usual assumption on homogeneity or exchangeability of the limiting system, and leave the door open for future study of more general non-dissipative regimes.
 
The proofs of Theorems \ref{thm:finite} and \ref{thm:infinite} start by reducing the analysis of $W_1(\emppre^n_t, \lawlimitmean_t)$ to that of $W_1(\emplimit^n_t, \lawlimitmean_t)$, where $\emplimit^n_t$ is the empirical measure of $\{\Xlimit_{i/n}(t):i=1,\dotsc,n\}$.
In the homogeneous setup \cite[Proposition 6.1]{BolleyGuillinVillani2007quantitative}, this reduction is done by a classic coupling argument.
However, due to the heterogeneity of the stochastic processes $\{\Xlimit_{i/n}\}$ and the presence of random interactions $\{\xi_{ij}^n\}$, such a coupling argument does not provide the desired reduction any more.
Instead, we carefully bound $W_1(\emppre^n_t, \lawlimitmean_t)$ by $W_1(\emplimit^n_t, \lawlimitmean_t)$ in a different way in Proposition \ref{prop:prep}, with certain new error terms consisting of $\{\Xlimit_{i/n}, \xi_{ij}^n\}$ and requiring new treatments.
Such new terms are further estimated in the exponential scale  in Proposition \ref{prop:p-n-eps}. 
In particular, the key ingredients for such estimates are exponential bounds obtained in Lemmas \ref{lem:Tmctil}, \ref{lem:Tmc} and \ref{lem:pn2}, using concentration inequalities and sub-Gaussian properties.
The distance $W_1(\emplimit^n_t, \lawlimitmean_t)$ is analyzed in Lemma \ref{lem:concentration-empirical} using exponential concentration bounds, for empirical measures of independent (but not necessarily identically distributed) random variables, established in Proposition \ref{prop:concentration-empirical} which is a heterogenous version of \cite[Theorem 2.1]{BolleyGuillinVillani2007quantitative} where the i.i.d.\ setup is analyzed.

\subsection{Organization}

The paper is organized as follows.
In Section \ref{sec:model} we state the space of graphons, the standing assumptions, and well-posedness of systems \eqref{eq:system} and \eqref{eq:system-n}.
The main results on concentration bounds are stated in Theorems \ref{thm:finite} and \ref{thm:infinite}.
Proofs of these two results are given in Section \ref{sec:pf}, with the key exponential estimates provided in Section \ref{sec:key}.
Finally Appendices \ref{sec:appendix-1}--\ref{sec:appendix-3} collect proofs of some auxiliary results.

We close this section by introducing some frequently used notation.

\subsection{Notation}
\label{sec:notation}
Given a Polish space $\Smb$, denote by $\Bmc(\Smb)$ the Borel $\sigma$-field. 
Let $\Pmc(\Smb)$ be the space of probability measures on $\Smb$ endowed with the topology of weak convergence.
Denote by $\Cmb([0,T]:\Rd)$ the space of continuous functions from $[0,T]$ to $\Rd$, endowed with the topology of uniform convergence.
We will use $C, C_1, C_2, \dotsc$, to denote various positive constants in the paper and $C(m)$ to emphasize the dependence on some parameter $m$.
The probability law of a random variable $X$ will be denoted by $\Lmc(X)$.
Expectations under $\Pmb$ will be denoted by $\Emb$.
To simplify the notation, we will usually write $\Emb[X^k]$ as $\Emb X^k$.
For vectors $x,y \in \Rd$, denote by $|x|$ the Euclidean norm and $x \cdot y$ the inner product.
Denote by $W_p$, $p \in \Nmb$, the Wasserstein-$p$ distance (cf.\ \cite[Chapter 6]{Villani2008optimal}) on $\Pmc(\Rmb^d)$:
\begin{align*}
	W_p(m_1,m_2) & := \left( \inf_\pi \int_{\Rmb^d \times \Rmb^d} |x-y|^p \, \pi(dx\,dy) \right)^{1/p}, \quad m_1,m_2 \in \Pmc(\Rmb^d),
\end{align*}
where the infimum is taken over all probability measures $\pi \in \Pmc(\Rmb^d \times \Rmb^d)$ with marginals $m_1$ and $m_2$, that is, $\pi(\cdot \times \Rmb^d) = m_1(\cdot)$ and $\pi(\Rmb^d \times \cdot) = m_2(\cdot)$.

\section{Model, assumptions, results}
\label{sec:model}

We follow the notation used in \cite[Chapters 7 and 8]{Lovasz2012large}.
Let $I := [0,1]$.
Denote by $\Gmc$ the space of all bounded measurable functions $G \colon I \times I \to \Rmb$.
A (directed) graphon $G$ is an element of $\Gmc$ with $0 \le G \le 1$.

Given a graphon $G \in \Gmc$ and a collection of initial distributions $\lawlimit(0) := (\lawlimit_u(0) \in \Pmc(\Rd) : u \in I)$,
recall the graphon particle system \eqref{eq:system} and the finite particle system \eqref{eq:system-n}.
The following assumptions will be made throughout the paper.

\vspace{.08in}
\noindent\textbf{Standing Assumptions:}
	\begin{itemize}
	\item 
		The map $I \ni u \mapsto \Lmc(\lawlimit_u(0)) \in \Pmc(\Rd)$ is measurable, 
		and $\sup_{u \in I} \int_{\Rd} e^{\theta_0 |x|^2} \,\lawlimit_{u,0}(dx) < \infty$ for some $\theta_0 > 0$.
	\item 
		The drift functions $f$ and $b$ are Lipschiz with Lipschitz constants $K_f$ and $K_b$, respectively, namely
		\begin{align*}
			|f(x_1)-f(x_2)| & \le K_f |x_1-x_2|, \quad \forall\, x_1,x_2 \in \Rd, \\
			|b(x_1,y_1)-b(x_2,y_2)| & \le K_b (|x_1-x_2| +|y_1-y_2|), \quad \forall\, x_1,x_2,y_1,y_2 \in \Rd.
		\end{align*}
	\item
		$G \in \Gmc$ is a directed graphon and $\xi_{ij}^n$ is the \textit{sampled graphon}, namely
		\begin{enumerate}[(i)]
		\item either $\xi_{ij}^n=G(\frac{i}{n},\frac{j}{n})$ for $i,j \in \{1,\dotsc,n\}$,
		\item or $\xi_{ij}^n=\text{Bernoulli}(G(\frac{i}{n},\frac{j}{n}))$ independently for $i,j \in \{1,\dotsc,n\}$, and independent of $\{\Xlimit_u(0),B_u : u \in I\}$.
		\end{enumerate}
	\item
		There exist some $K_G \in (0,\infty)$ and a finite collection of disjoint intervals $\{I_i : i=1,\dotsc,N\}$ for some $N \in \Nmb$, such that $\cup_{i=1}^N I_i = I$ and
		\begin{align*}
			W_2(\lawlimit_{u_1}(0),\lawlimit_{u_2}(0)) & \le K_G |u_1-u_2|, \quad u_1,u_2 \in I_i, \quad i \in \{1,\dotsc,N\}, \\
			|G(u_1,v_1)-G(u_2,v_2)| & \le K_G (|u_1-u_2|+|v_1-v_2|), \: (u_1,v_1),(u_2,v_2) \in I_i \times I_j, \: i,j \in \{1,\dotsc,N\}.
		\end{align*}
	\end{itemize} 	

We note that, since $f$ is Lipschitz, the quantity
\begin{equation*}
	c_0 := \inf_{x_1,x_2 \in \Rd, x_1 \ne x_2} \frac{-(x_1-x_2) \cdot (f(x_1)-f(x_2))}{|x_1-x_2|^2}
\end{equation*}
is well-defined, and $|c_0| \le K_f$.
From this we have 
\begin{equation}
	\label{eq:c0}
	(x_1-x_2) \cdot (f(x_1)-f(x_2)) \le -c_0 |x_1-x_2|^2, \quad \forall\, x_1,x_2 \in \Rd.
\end{equation}
Let
\begin{equation}
	\label{eq:kappa}
	\kappa := c_0-2K_b.
\end{equation}
For several long-time results the following dissipativity assumption will be made.

\begin{Condition}
	\label{cond:strong}
	The coefficient $f$ is dissipative, in the sense that $\kappa>0$.
\end{Condition}

\begin{Remark}
	A common example of $b$ and $f$ satisfying Condition \ref{cond:strong} is linear (as in the study of linear quadratic graphon mean-field games in, e.g., \cite{GaoTchuendomCaines2020linear}) and mean-reverting:
	$$f(x) + b(x,y)=-c_1x+c_2y, \quad \text{for some } c_1 > c_2 > 0.$$
	In particular, the choice of $f(x)=-(c_1+c_2)x$ and $b(x,y)=c_2(x+y)$ satisfies Condition \ref{cond:strong} since $c_0=c_1+c_2 > 2c_2=2K_b$.
	For a more general linear mean-reverting example, see \cite[Example 3.1]{BayraktarWu2020stationarity}.
\end{Remark}

The following result gives well-posedness of systems \eqref{eq:system} and \eqref{eq:system-n}. 
The proof is standard (see e.g.\ the homogeneous setup \cite{Sznitman1991} and the heterogeneous setup \cite{BayraktarChakrabortyWu2020graphon} for part (i), and \cite[Theorems 5.2.5 and 5.2.9]{KaratzasShreve1991brownian} for part (ii)) and hence is omitted.

\begin{Proposition}
	\phantomsection
	\label{prop:well-posedness}
	\begin{enumerate}[(i)] 
	\item
		There exists a unique pathwise solution to \eqref{eq:system}.
		For every $T < \infty$, the map $I \ni u \mapsto \lawlimit_u \in \Pmc(\Cmb([0,T]:\Rd))$ is measurable and 
		\begin{equation*}
			\sup_{u \in I} \sup_{t \in [0,T]} \Emb \left[ |\Xlimit_u(t)|^4 \right] < \infty.
		\end{equation*}
	\item 
		There exists a unique pathwise solution to \eqref{eq:system-n}.
		Also for every $T < \infty$, 
		\begin{equation*}
			\max_{i=1,\dotsc,n} \sup_{t \in [0,T]} \Emb \left[ |X_i^n(t)|^4 \right] < \infty.
		\end{equation*}
	\end{enumerate}	
\end{Proposition}

\begin{Remark}	
	\label{rmk:formulation}
	The formulation \eqref{eq:system} with Brownian motions $\{B_u : u \in I\}$ is used to describe the nonlinear McKean--Vlasov type stochastic processes $\{\Xlimit_u : u \in I\}$ and emphasize their independence.
	However, results in this work are on measures of $\{\Xlimit_u : u \in I\}$.
	Therefore, one may work with the following equivalent formulation with a single Brownian motion $\{B(t) : t \ge 0\}$
	\begin{equation*}
		\Xlimit_u(t) = \Xlimit_u(0) + \int_0^t \left( f(\Xlimit_u(s)) + \int_0^1 \int_{\Rmb^d} b(\Xlimit_u(s),x)G(u,v) \,\lawlimit_{v,s}(dx)\,dv\right)ds + B(t),
	\end{equation*}
	with $\lawlimit_{v,s}=\Lmc(\Xlimit_v(s))$ and these laws are independent for $v \in I$; see e.g.\ \cite{BetCoppiniNardi2020weakly,Lucon2020quenched,BayraktarWuZhang2022propagation}.	
	In fact, one may just work with the spatially extended law $\lawlimit_t(du\,dx) := \lawlimit_{u,t}(dx)\,du$ and the nonlinear Fokker--Planck equations
	\begin{align*}
		\lan \lawlimit_{u,t}, h \ran & = \lan \lawlimit_{u,0}, h \ran + \int_0^t \left\lan \lawlimit_{u,s}, \nabla h(\cdot) \cdot \left( f(\cdot) + \int_{I \times \Rd} b(\cdot,x)G(u,v) \,\lawlimit_{s}(dv\,dx)\right) \right\ran ds \\
		& \quad + \int_0^t \left\lan \lawlimit_{u,s}, \half \Delta h \right\ran ds
	\end{align*}
	for $u \in I$ and regular test functions $h$; see e.g.\ \cite{BetCoppiniNardi2020weakly,Lucon2020quenched}.
\end{Remark}


Let
\begin{equation*}
	\emppre^n(t) := \frac{1}{n} \sum_{i=1}^n \delta_{X_i^n(t)},
	\quad \emplimit^n(t) := \frac{1}{n} \sum_{i=1}^n \delta_{\Xlimit_\frac{i}{n}(t)}, \quad
	\lawlimitmean(t) := \int_I \lawlimit_{u,t} \, du.
\end{equation*}
We first have the following result on the trajectory of Wasserstein distances over finite time horizon.
Note that Condition \ref{cond:strong} is not used.

\begin{Theorem}
	\label{thm:finite}
	Fix $T \in (0,\infty)$.
	Then there exists some $K \in (0,\infty)$ such that for any $d'>d$, there exist $C =C_T \in (0,\infty)$ and $N_0 \in \Nmb$ such that
	\begin{equation}
		\label{eq:finite}
		\Pmb(\sup_{0 \le t \le T} W_1(\emppre^n_t, \lawlimitmean_t) > \varepsilon) \le C(1+\varepsilon^{-2}) \exp(-K\sqrt{n}\varepsilon)
	\end{equation}
	for all $\varepsilon > 0$ and $n \ge N_0 \max(\varepsilon^{-(d'+2)},1)$.
	If in addition
	\begin{equation}
		\label{eq:condition-special}
		\xi_{ij}^n=G(\frac{i}{n},\frac{j}{n}) \mbox{ or } b(\cdot,0) \mbox{ is bounded},
	\end{equation}
	then $\sqrt{n}\varepsilon$ in \eqref{eq:finite} could be improved to $n\varepsilon^2$.
\end{Theorem}

\begin{Remark}
	The technical assumption \eqref{eq:condition-special} will be used in Lemma \ref{lem:Tmctil} to derive certain sub-Gaussian bounds. The assumption on $b$ holds in particular if $b(x,y)=b(y)$ is $x$-independent or if $b$ is bounded.
\end{Remark}

When the convexity property in Condition \ref{cond:strong} is satisfied, the following uniform-in-time marginal concentration bound holds.

\begin{Theorem}
	\label{thm:infinite}
	Suppose Condition \ref{cond:strong} holds.
	Then there exists some $K \in (0,\infty)$ such that for any $d'>d$, there exist $C \in (0,\infty)$ and $N_0 \in \Nmb$ such that
	\begin{equation}
		\label{eq:infinite}
		\sup_{t \ge 0} \Pmb(W_1(\emppre^n_t, \lawlimitmean_t) > \varepsilon) \le C\exp(-K\sqrt{n}\varepsilon)
	\end{equation}
	for	all $\varepsilon>0$ and $n \ge N_0 \max(\varepsilon^{-(d'+2)},1)$.
	If in addition \eqref{eq:condition-special} holds, then $\sqrt{n}\varepsilon$ in \eqref{eq:infinite} could be improved to $n\varepsilon^2$.
\end{Theorem}

\begin{Remark} 
	\label{rmk:improvement}
	In the classic mean-field regime, namely when $G \equiv 1$, one has $\xi_{ij}^n \equiv 1$ so that the condition \eqref{eq:condition-special} holds and the exponential rates in \eqref{eq:finite} and \eqref{eq:infinite} are $n\varepsilon^2$. 
	From this we observe that Theorems \ref{thm:finite} and \ref{thm:infinite} are consistent with the mean-field results in \cite[Theorems 2.9 and 2.12] {BolleyGuillinVillani2007quantitative}.
	In fact, the pre-exponential term $C$ in \eqref{eq:infinite} does not depend on $\varepsilon$, while the corresponding term in \cite[Theorem 2.12] {BolleyGuillinVillani2007quantitative} is $C(1+\varepsilon^{-2})$.
	This improvement is due to a different argument (Proposition \ref{prop:p-n-eps}) that does not use time discretization as in \cite[Section 7.2]{BolleyGuillinVillani2007quantitative}.
\end{Remark}

Using Theorem \ref{thm:infinite} we could obtain an exponential concentration bound for numerical reconstruction of the invariant measure.


\begin{Corollary}
	\label{cor:invariant}
	Suppose Condition \ref{cond:strong} holds.
	Then there exist a unique probability measure $\lawlimitmean_\infty \in \Pmc(\Rd)$ and some $\varepsilon_0,T_0,K \in (0,\infty)$ such that for any $d'>d$, there exist $C \in (0,\infty)$ and $N_0 \in \Nmb$ such that
	\begin{equation}
		\label{eq:invariant}
		\sup_{t \ge T_0\log(\varepsilon_0/\varepsilon)} \Pmb(W_1(\emppre^n_t, \lawlimitmean_\infty) > \varepsilon) \le C \exp(-K\sqrt{n}\varepsilon)
	\end{equation}
	for all $\varepsilon > 0$ and $n \ge N_0 \max(\varepsilon^{-(d'+2)},1)$.
	If in addition \eqref{eq:condition-special} holds, then $\sqrt{n}\varepsilon$ in \eqref{eq:invariant} could be improved to $n\varepsilon^2$.	
\end{Corollary}

\begin{proof}
	By the exponential ergodicity of $\lawlimitmean_t$ shown in
	\cite[Theorem 3.1]{BayraktarWu2020stationarity}, there exist a unique probability measure $\lawlimitmean_\infty \in \Pmc(\Rd)$ and some $\Ctil \in (0,\infty)$ such that 
	$$W_2(\lawlimitmean_t,\lawlimitmean_\infty) \le \Ctil e^{-\kappa t/2}, \quad t \ge 0.$$
	Now let $T_0=2/\kappa$ and $\varepsilon_0=2\Ctil$.
	Then for $t \ge T_0\log(\varepsilon_0/\varepsilon)$, we have $$W_1(\lawlimitmean_t,\lawlimitmean_\infty) \le W_2(\lawlimitmean_t,\lawlimitmean_\infty) \le \Ctil e^{-\kappa t/2} \le \varepsilon/2$$
	and hence
	\begin{equation*}
		\Pmb(W_1(\emppre^n_t, \lawlimitmean_\infty) > \varepsilon) \le \Pmb(W_1(\emppre^n_t, \lawlimitmean_t) > \varepsilon/2).
	\end{equation*}
	The result then follows from Theorem \ref{thm:infinite}.
\end{proof}

\section{Proofs}
\label{sec:pf}

In this section we prove Theorems \ref{thm:finite} and \ref{thm:infinite} on the concentration inequalities of $W_1(\emppre^n_t, \lawlimitmean_t)$.
We first state some properties of the limiting system in Section \ref{sec:property}.
Using these, in Section \ref{sec:transfer}, we reduce the analysis of $W_1(\emppre^n_t, \lawlimitmean_t)$ to that of $W_1(\emplimit^n_t, \lawlimitmean_t)$ and other quantities only involving independent but heterogeneous processes.
Such quantities are carefully analyzed in Section \ref{sec:key}.
In Section \ref{sec:complete-pf} we obtain the concentration bounds for $W_1(\emplimit^n_t, \lawlimitmean_t)$ and complete the proofs of Theorems \ref{thm:infinite} and \ref{thm:finite}.

\subsection{Properties of the limiting system}
\label{sec:property}

In this section we state some properties of the limiting system that will be used later.
Proofs will be given in Appendix \ref{sec:appendix-1}.

\begin{Lemma}{\cite[Theorem 2.1]{BayraktarChakrabortyWu2020graphon}}
	\label{lem:Lipschitz}
	For every $T \in (0,\infty)$, there exists some $C(T) \in (0,\infty)$ such that
	$$\sup_{t \in [0,T]} W_2(\lawlimit_{u_1}(t),\lawlimit_{u_2}(t)) \le C(T) |u_1-u_2|$$
	whenever $u_1,u_2 \in I_i$ for some $i \in \{1,\dotsc,N\}$.
\end{Lemma}

\begin{Lemma}{\cite[Proposition 3.1 and Corollary 3.1]{BayraktarWu2020stationarity}}
	\label{lem:Lipschitz-special}
	Suppose Condition \ref{cond:strong} holds.
	Then
	$$\sup_{u \in I} \sup_{t \ge 0} \Emb \left[ |\Xlimit_u(t)|^4 \right] < \infty.$$
	Moreover, there exists some $C \in (0,\infty)$ such that
	$$\sup_{t \ge 0} W_2(\lawlimit_{u_1}(t),\lawlimit_{u_2}(t)) \le C |u_1-u_2|$$
	whenever $u_1,u_2 \in I_i$ for some $i \in \{1,\dotsc,N\}$.
\end{Lemma}


The following result on the square exponential moments of $\Xlimit_u$ is a generalization of \cite[Proposition 4.1]{BolleyGuillinVillani2007quantitative} and \cite[Proposition 4.3]{BudhirajaFan2017}. 
The proof is provided in Appendix for completeness.
Recall $\theta_0$ in the standing assumption and $\kappa$ in \eqref{eq:kappa}.

\begin{Lemma}
	\phantomsection
	\label{lem:square-exponential-moment}
	\begin{enumerate}[(i)]
	\item For every $T \in (0,\infty)$, there exists some $\theta_T \in (0,\infty)$ such that 
	\begin{equation*}
		\sup_{u \in I} \Emb[ \sup_{t \in [0,T]} e^{\theta_T |\Xlimit_u(t)|^2}] < \infty.
	\end{equation*}
	\item Suppose Condition \ref{cond:strong} holds. Then for any $\theta \in (0,(\kappa \wedge \theta_0)/4)$, we have
	\begin{equation*}
		\sup_{t \ge 0} \sup_{u \in I} \Emb[ e^{\theta |\Xlimit_u(t)|^2}] < \infty.
	\end{equation*}
	\end{enumerate}
\end{Lemma}


The following two lemmas on the time-regularity of $\Xlimit_u$, $\lawlimitmean$ and $\emplimit^n$ are generalizations of results in \cite[Section 5]{BolleyGuillinVillani2007quantitative} and \cite[Section 4.3.3]{BudhirajaFan2017}.
We provide the proofs in Appendix for completeness.

\begin{Lemma}
	\label{lem:regularity-moment}
	For every $T \in (0,\infty)$, there exist some $C(T),\theta_T \in (0,\infty)$ such that, for all $s,t,t_0,\Delta \in [0,T]$ and $u \in I$,
	\begin{align*}
		\Emb |\Xlimit_u(t)-\Xlimit_u(s)|^4 & \le C(T)|t-s|^2, \\
		\Emb \left[ \sup_{t_0 \le s \le t \le t_0+\Delta} \exp \left(\theta_T |\Xlimit_u(t)-\Xlimit_u(s)|^2 \right) \right] & \le 1+C(T)\Delta.
	\end{align*}
\end{Lemma}


\begin{Lemma}
	\label{lem:regularity-n}
	For every $T \in (0,\infty)$, there exist some $C(T), \theta_T \in (0,\infty)$ such that
	$$\Pmb \left(\sup_{t_0 \le s,t \le t_0+\Delta} W_1(\emplimit^n_s, \emplimit^n_t) > \varepsilon\right) \le \exp \left( -n \left[ \theta_T\varepsilon^2-C(T)\Delta\right] \right)$$
	for all $t_0 \in [0,T]$, $\Delta \in [0,T]$ and $\varepsilon \in (0,\infty)$.
\end{Lemma}

The following result is an immediate consequence of Lemma \ref{lem:regularity-moment}.

\begin{Lemma}
	\label{lem:regularity}
	For every $T \in (0,\infty)$, there exists some $C(T) \in (0,\infty)$ such that
	$$W_1(\lawlimitmean_t, \lawlimitmean_s) \le C(T)|t-s|^{1/2}, \quad \forall s,t \in [0,T].$$
\end{Lemma}

\subsection{Bounds in terms of empirical measures of independent variables}
\label{sec:transfer}

The following proposition is a generalization of \cite[Proposition 6.1]{BolleyGuillinVillani2007quantitative} and reduces the analysis of $W_1(\emppre^n_t, \lawlimitmean_t)$ to quantities only in terms of independent processes.

\begin{Proposition}
	\label{prop:prep}
	Let
	\begin{equation}
		\label{eq:Rmctil}		
		\Rmc^{n,i}_t:= \frac{1}{n} \sum_{j=1}^n \left( \xi_{ij}^n b(\Xlimit_{\frac{i}{n}}(t),\Xlimit_{\frac{j}{n}}(t)) - \int_{\Rmb^d} b(\Xlimit_{\frac{i}{n}}(t),x)G(\frac{i}{n},\frac{j}{n}) \,\lawlimit_{\frac{j}{n},t}(dx) \right).
	\end{equation}
	\begin{enumerate}[(i)]
	\item 
		For every $T \in (0,\infty)$, there exists some $C(T) \in (0,\infty)$ such that
		\begin{equation*}
			W_1(\emppre^n_t, \lawlimitmean_t) \le W_1(\emplimit^n_t, \lawlimitmean_t) + C(T) \left( \int_0^t e^{-\kappa (t-s)} \frac{1}{n} \sum_{i=1}^n \left| |\Rmc_s^{n,i}| + \frac{1}{n} \left( 1+|\Xlimit_{\frac{i}{n}}(s)| \right) \right|^2 ds \right)^{1/2}
		\end{equation*}
		for all $t \in [0,T]$.
	\item
		Suppose Condition \ref{cond:strong} holds. 
		Then there exists some $C \in (0,\infty)$, independent of $t \ge 0$, such that
		\begin{equation*}
			W_1(\emppre^n_t, \lawlimitmean_t) \le W_1(\emplimit^n_t, \lawlimitmean_t) + C \left( \int_0^t e^{-\kappa (t-s)} \frac{1}{n} \sum_{i=1}^n \left| |\Rmc_s^{n,i}| + \frac{1}{n} \left( 1+|\Xlimit_{\frac{i}{n}}(s)| \right) \right|^2 ds \right)^{1/2}
		\end{equation*}
		for all $t \ge 0$.
	\end{enumerate}	
\end{Proposition}

\begin{proof}
	We first prove part (ii).
	From \eqref{eq:system} and \eqref{eq:system-n} we have
	\begin{align*}
		& d\left| X_i^n(t)-\Xlimit_{\frac{i}{n}}(t) \right|^2 \\
		& = 2\left( X_i^n(t)-\Xlimit_{\frac{i}{n}}(t) \right) \cdot \left( f(X_i^n(t))-f(\Xlimit_{\frac{i}{n}}(t)) \right) dt \\
		& \quad + 2\left( X_i^n(t)-\Xlimit_{\frac{i}{n}}(t) \right) \cdot \left( \frac{1}{n} \sum_{j=1}^n \xi_{ij}^n b(X_i^n(t),X_j^n(t)) - \int_I \int_{\Rmb^d} b(\Xlimit_{\frac{i}{n}}(t),x)G(\frac{i}{n},v)\,\lawlimit_{v,t}(dx)\,dv \right) dt.
	\end{align*}
	By \eqref{eq:c0},
	\begin{align*}
		\left( X_i^n(t)-\Xlimit_{\frac{i}{n}}(t) \right) \cdot \left( f(X_i^n(t))-f(\Xlimit_{\frac{i}{n}}(t)) \right) \le -c_0 \left|X_i^n(t)-\Xlimit_{\frac{i}{n}}(t)\right|^2.
	\end{align*}
	By adding and subtracting terms, we have
	\begin{align*}
		& \left( X_i^n(t)-\Xlimit_{\frac{i}{n}}(t) \right) \cdot \left( \frac{1}{n} \sum_{j=1}^n \xi_{ij}^n b(X_i^n(t),X_j^n(t)) - \int_I \int_{\Rmb^d} b(\Xlimit_{\frac{i}{n}}(t),x)G(\frac{i}{n},v)\,\lawlimit_{v,t}(dx)\,dv \right) \\
		& = \left( X_i^n(t)-\Xlimit_{\frac{i}{n}}(t) \right) \cdot \left[ \frac{1}{n} \sum_{j=1}^n \xi_{ij}^n \left( b(X_i^n(t),X_j^n(t)) - b(\Xlimit_{\frac{i}{n}}(t),\Xlimit_{\frac{j}{n}}(t)) \right) \right] \\
		& \quad + \left( X_i^n(t)-\Xlimit_{\frac{i}{n}}(t) \right) \cdot \left[ \frac{1}{n} \sum_{j=1}^n \left( \xi_{ij}^n b(\Xlimit_{\frac{i}{n}}(t),\Xlimit_{\frac{j}{n}}(t)) - \int_{\Rmb^d} b(\Xlimit_{\frac{i}{n}}(t),x)G(\frac{i}{n},\frac{j}{n}) \,\lawlimit_{\frac{j}{n},t}(dx) \right) \right] \\
		& \quad + \left( X_i^n(t)-\Xlimit_{\frac{i}{n}}(t) \right) \cdot \left( \frac{1}{n} \sum_{j=1}^n \int_{\Rmb^d} b(\Xlimit_{\frac{i}{n}}(t),x)G(\frac{i}{n},\frac{j}{n}) \,\lawlimit_{\frac{j}{n},t}(dx) - \int_I \int_{\Rmb^d} b(\Xlimit_{\frac{i}{n}}(t),x)G(\frac{i}{n},v)\,\lawlimit_{v,t}(dx)\,dv \right) \\
		& =: \Rmctil_t^{n,i,1} + \Rmctil_t^{n,i,2} + \Rmctil_t^{n,i,3}.
	\end{align*}
	For $\Rmctil_t^{n,i,1}$, using the Lipschitz property of $b$, we have
	\begin{align*}
		\Rmctil_t^{n,i,1} \le \frac{K_b}{n} \sum_{j=1}^n \left| X_i^n(t)-\Xlimit_{\frac{i}{n}}(t) \right| \left( \left| X_i^n(t)-\Xlimit_{\frac{i}{n}}(t) \right| + \left| X_j^n(t)-\Xlimit_{\frac{j}{n}}(t) \right| \right).
	\end{align*}
	For $\Rmctil_t^{n,i,2}$, from \eqref{eq:Rmctil} we have
	\begin{equation*}
		\Rmctil_t^{n,i,2} \le \left| X_i^n(t)-\Xlimit_{\frac{i}{n}}(t) \right| |\Rmc_s^{n,i}|.
	\end{equation*}
	For $\Rmctil_t^{n,i,3}$, using Lemma \ref{lem:Lipschitz-special}, we have
	\begin{align*}
		\Rmctil_t^{n,i,3} & = \left( X_i^n(t)-\Xlimit_{\frac{i}{n}}(t) \right) \cdot \int_I \int_{\Rmb^d} b(\Xlimit_{\frac{i}{n}}(t),x) \left( G(\frac{i}{n},\frac{\lceil nv \rceil}{n}) \,\lawlimit_{\frac{\lceil nv \rceil}{n},t}(dx) - G(\frac{i}{n},v)\,\lawlimit_{v,t}(dx) \right) dv \\
		& \le \frac{C_1}{n} |X_i^n(t)-\Xlimit_{\frac{i}{n}}(t)| \left( 1+|\Xlimit_{\frac{i}{n}}(t)| \right).
	\end{align*}
	
	Taking the average over $i$ and using \eqref{eq:kappa}, we get
	\begin{align*}
		\frac{d}{dt} \frac{1}{n} \sum_{i=1}^n \left| X_i^n(t)-\Xlimit_{\frac{i}{n}}(t) \right|^2 & \le -2\kappa \frac{1}{n} \sum_{i=1}^n \left| X_i^n(t)-\Xlimit_{\frac{i}{n}}(t) \right|^2 \\
		& \quad + 2 \left( \frac{1}{n} \sum_{i=1}^n \left| X_i^n(t)-\Xlimit_{\frac{i}{n}}(t) \right|^2 \right)^{1/2} \left( \frac{1}{n} \sum_{i=1}^n \left| |\Rmc_t^{n,i}| + \frac{C_1}{n} \left( 1+|\Xlimit_{\frac{i}{n}}(t)| \right) \right|^2 \right)^{1/2}.
	\end{align*}
	So
	\begin{align*}
		& \frac{d}{dt} \left( \frac{1}{n} \sum_{i=1}^n \left| X_i^n(t)-\Xlimit_{\frac{i}{n}}(t) \right|^2 \right)^{1/2} \\
		& = \frac{1}{2} \left( \frac{1}{n} \sum_{i=1}^n \left| X_i^n(t)-\Xlimit_{\frac{i}{n}}(t) \right|^2 \right)^{-1/2} \frac{d}{dt} \frac{1}{n} \sum_{i=1}^n \left| X_i^n(t)-\Xlimit_{\frac{i}{n}}(t) \right|^2 \\
		& \le -\kappa \left( \frac{1}{n} \sum_{i=1}^n \left| X_i^n(t)-\Xlimit_{\frac{i}{n}}(t) \right|^2 \right)^{1/2} + \left( \frac{1}{n} \sum_{i=1}^n \left| |\Rmc_t^{n,i}| + \frac{C_1}{n} \left( 1+|\Xlimit_{\frac{i}{n}}(t)| \right) \right|^2 \right)^{1/2}.
	\end{align*}	
	From this we have 
	\begin{align}
		\left( \frac{1}{n} \sum_{i=1}^n \left| X_i^n(t)-\Xlimit_{\frac{i}{n}}(t) \right|^2 \right)^{1/2} & \le \int_0^t e^{-\kappa (t-s)} \left( \frac{1}{n} \sum_{i=1}^n \left| |\Rmc_s^{n,i}| + \frac{C_1}{n} \left( 1+|\Xlimit_{\frac{i}{n}}(s)| \right) \right|^2 \right)^{1/2} ds \notag \\
		& \le C_2 \left( \int_0^t e^{-\kappa (t-s)} \frac{1}{n} \sum_{i=1}^n \left| |\Rmc_s^{n,i}| + \frac{1}{n} \left( 1+|\Xlimit_{\frac{i}{n}}(s)| \right) \right|^2 ds \right)^{1/2}, \label{eq:prep-pf-1}
	\end{align}
	where the last line follows from Holder's inequality and Condition \ref{cond:strong}. 
	Noting that
	\begin{equation*}
		W_1(\emppre^n_t, \emplimit^n_t) \le W_2(\emppre^n_t, \emplimit^n_t) \le \left( \frac{1}{n} \sum_{i=1}^n \left| X_i^n(t)-\Xlimit_{\frac{i}{n}}(t) \right|^2 \right)^{1/2}
	\end{equation*}
	and
	\begin{equation*}
		W_1(\emppre^n_t, \lawlimitmean_t) \le W_1(\emplimit^n_t, \lawlimitmean_t) + W_1(\emppre^n_t, \emplimit^n_t),
	\end{equation*}
	we have the desired result for part (ii).
	
	The proof of part (i) is the same as that of part (ii), except that Lemma \ref{lem:Lipschitz-special} and all the relevant constants (such as $C_1$) are to be replaced by Lemma \ref{lem:Lipschitz} and $T$-dependent constants (such as $C_1(T)$).
\end{proof}

\subsection{Exponential estimates}
\label{sec:key}

Recall $\Rmc^{n,i}_t$ introduced in \eqref{eq:Rmctil}.
In this section we will provide an exponential bound for the probability
\begin{equation}
	\label{eq:pn}
	p(n,t,\varepsilon) := \Pmb\left( \int_0^t e^{-\kappa (t-s)} \frac{1}{n} \sum_{i=1}^n \left| |\Rmc_s^{n,i}| + \frac{1}{n} \left( 1+|\Xlimit_{\frac{i}{n}}(s)| \right) \right|^2 ds > \varepsilon \right)
\end{equation}
for $t \ge 0 $ and $\varepsilon \in (0,\infty)$.

\begin{Proposition}
	\phantomsection
	\label{prop:p-n-eps}
	\begin{enumerate}[(i)]
	\item 
		There exist some $C(T), K_T \in (0,\infty)$ and $N_T \in \Nmb$ such that
		\begin{equation}
			\label{eq:pnt}
			\sup_{t \in [0,T]} p(n,t,\varepsilon) \le C(T)e^{-K_T \sqrt{n\varepsilon}}
		\end{equation}
		for all $\varepsilon > 0$ and $n \ge N_T /\sqrt{\varepsilon}$.
		If in addition \eqref{eq:condition-special} holds, then $\sqrt{n\varepsilon}$ in \eqref{eq:pnt} could be improved to $n\varepsilon$.
	\item 
		Suppose Condition \ref{cond:strong} holds.
		Then there exist some $C,K \in (0,\infty)$ and $N_0 \in \Nmb$ such that
		\begin{equation}
			\label{eq:pnt-special}
			\sup_{t \ge 0} p(n,t,\varepsilon) \le Ce^{-K \sqrt{n\varepsilon}}
		\end{equation}
		for all $\varepsilon > 0$ and $n \ge N_0 /\sqrt{\varepsilon}$.
		If in addition \eqref{eq:condition-special} holds, then $\sqrt{n\varepsilon}$ in \eqref{eq:pnt-special} could be improved to $n\varepsilon$.
	\end{enumerate}

\end{Proposition}

\begin{proof}
	We first prove part (ii).
	Using the union bound, we have
	\begin{equation*}
		p(n,t,\varepsilon) \le \sum_{i=1}^n \Pmb\left( \int_0^t e^{-\kappa  (t-s)} \left| |\Rmc_s^{n,i}| + \frac{1}{n} \left( 1+|\Xlimit_{\frac{i}{n}}(s)| \right) \right|^2 ds > \varepsilon \right).
	\end{equation*}
	Now we fix $i \in \Nmb$.
	Writing $\Rmc_s^{n,i}=\Tmctil_s^{n,i}+\Tmc_s^{n,i}$, where
	\begin{align}
		\Tmctil_s^{n,i} & := \frac{1}{n} \sum_{j=1}^n \left( \xi_{ij}^n - G(\frac{i}{n},\frac{j}{n}) \right) b(\Xlimit_{\frac{i}{n}}(s),0), \label{eq:Tmctil} \\
		\Tmc_s^{n,i} & := \frac{1}{n} \sum_{j=1}^n \left[ \left( \xi_{ij}^n - G(\frac{i}{n},\frac{j}{n}) \right) \left( b(\Xlimit_{\frac{i}{n}}(s),\Xlimit_{\frac{j}{n}}(s)) - b(\Xlimit_{\frac{i}{n}}(s),0) \right) \right. \notag \\
		& \qquad \left. + G(\frac{i}{n},\frac{j}{n}) \left( b(\Xlimit_{\frac{i}{n}}(s),\Xlimit_{\frac{j}{n}}(s)) - \int_{\Rmb^d} b(\Xlimit_{\frac{i}{n}}(s),x) \,\lawlimit_{\frac{j}{n},s}(dx) \right) \right], \label{eq:Tmc}
	\end{align}	
	we have
	$$\left| |\Rmc_s^{n,i}| + \frac{1}{n} \left( 1+|\Xlimit_{\frac{i}{n}}(s)| \right) \right|^2 \le 4 \left( |\Tmctil_s^{n,i}|^2 + |\Tmc_s^{n,i}|^2 + \frac{1}{n^2} + \frac{1}{n^2} |\Xlimit_{\frac{i}{n}}(s)|^2 \right).$$
	Therefore
	\begin{align*}
		& \Pmb\left( \int_0^t e^{-\kappa  (t-s)} \left| |\Rmc_s^{n,i}| + \frac{1}{n} \left( 1+|\Xlimit_{\frac{i}{n}}(s)| \right) \right|^2 ds > \varepsilon \right) \\
		& \le \Pmb\left( \int_0^t e^{-\kappa  (t-s)} |\Tmctil_s^{n,i}|^2 ds > \varepsilon/16 \right) + \Pmb\left( \int_0^t e^{-\kappa  (t-s)} |\Tmc_s^{n,i}|^2 ds > \varepsilon/16 \right) \\
		& \quad + \Pmb\left( \int_0^t e^{-\kappa  (t-s)} \frac{1}{n^2} ds > \varepsilon/16 \right) + \Pmb\left( \int_0^t e^{-\kappa  (t-s)} \frac{1}{n^2} |\Xlimit_{\frac{i}{n}}(s)|^2\, ds > \varepsilon/16 \right) \\
		& =: p(n,t,\varepsilon,i,1)+p(n,t,\varepsilon,i,2)+p(n,t,\varepsilon,i,3)+p(n,t,\varepsilon,i,4),
	\end{align*}
	and hence 
	\begin{equation}
		\label{eq:pn-total}
		p(n,t,\varepsilon) \le \sum_{i=1}^n \left[ p(n,t,\varepsilon,i,1)+p(n,t,\varepsilon,i,2)+p(n,t,\varepsilon,i,3)+p(n,t,\varepsilon,i,4) \right].
	\end{equation}
	
	Next we analyze each term.
	For $p(n,t,\varepsilon,i,4)$, applying Markov's inequality and Jensen's inequality, we have
	\begin{align*}
		p(n,t,\varepsilon,i,4)
		& \le e^{-\theta \varepsilon/16} \Emb \left[\exp \left( \theta \int_0^t e^{-\kappa  (t-s)} \frac{1}{n^2} |\Xlimit_{\frac{i}{n}}(s)|^2\, ds \right) \right] \\
		& \le e^{-\theta \varepsilon/16} \Emb \left[\int_0^t C_1e^{-\kappa  (t-s)} \exp \left( \frac{C_1\theta}{n^2} |\Xlimit_{\frac{i}{n}}(s)|^2 \right) ds \right]
	\end{align*}
	for each $\theta > 0$.
	Taking $\theta=n^2(\kappa \wedge \theta_0)/8C_1$ and using Lemma \ref{lem:square-exponential-moment}(ii), we have
	$$p(n,t,\varepsilon,i,4) \le C_2 e^{-C_3 n^2 \varepsilon}.$$	
	For $p(n,t,\varepsilon,i,3)$, we have
	$$p(n,t,\varepsilon,i,3)=0$$
	whenever $n>C_4/\sqrt{\varepsilon}$.	
	For $p(n,t,\varepsilon,i,2)$, we will show in Lemma \ref{lem:pn2}(ii) below that
	\begin{equation*}
		p(n,t,\varepsilon,i,2) \le C_5e^{-C_6 n \varepsilon}.
	\end{equation*}	
	Combining these with \eqref{eq:pn-total} and Lemma \ref{lem:Tmctil}(ii) below  gives
	\begin{equation*}
		p(n,t,\varepsilon) \le C_7ne^{-C_8\sqrt{n \varepsilon}} \le C_9e^{-C_{10} \sqrt{n \varepsilon}},
	\end{equation*}
	namely \eqref{eq:pnt-special} holds, for $n>C_{11}/\sqrt{\varepsilon}$.
	It also follows from Lemma \ref{lem:Tmctil}(ii) that if in addition \eqref{eq:condition-special} holds, then the above $\sqrt{n \varepsilon}$ could be improved to $n\varepsilon$.
	This gives part (ii).
	
	The proof of part (i) is the same as that of part (ii), except that Lemmas \ref{lem:square-exponential-moment}(ii), \ref{lem:Tmctil}(ii), \ref{lem:pn2}(ii) and all the relevant constants (such as $C_1$) are to be replaced by Lemmas \ref{lem:square-exponential-moment}(i), \ref{lem:Tmctil}(i), \ref{lem:pn2}(i), and $T$-dependent constants (such as $C_1(T)$).
\end{proof}

In the following lemma we estimate the probability $p(n,t,\varepsilon,i,1)$.
Recall that
\begin{equation*}
	p(n,t,\varepsilon,i,1) = \Pmb\left( \int_0^t e^{-\kappa  (t-s)} |\Tmctil_s^{n,i}|^2 ds > \varepsilon/16 \right),
\end{equation*}
where $\Tmctil_s^{n,i}$ was introduced in \eqref{eq:Tmctil}.

\begin{Lemma}
	\phantomsection
	\label{lem:Tmctil}
	\begin{enumerate}[(i)]
	\item 
		For every $T \in (0,\infty)$, there exist some $C(T), K_T \in (0,\infty)$ such that
		\begin{equation}
			\label{eq:Tmctil-sub-Gaussian-1}
			\max_{i=1,\dotsc,n} \sup_{t \in [0,T]} p(n,t,\varepsilon,i,1) \le C(T)e^{-K_T\sqrt{n\varepsilon}}, \quad \forall n \in \Nmb, \: \varepsilon>0.
		\end{equation}	
		If in addition \eqref{eq:condition-special} holds, then $\sqrt{n\varepsilon}$ in \eqref{eq:Tmctil-sub-Gaussian-1} could be improved to $n\varepsilon$. 
	\item 
		Suppose that Condition \ref{cond:strong} holds. Then there exist some $C, K \in (0,\infty)$ such that
		\begin{equation}
			\label{eq:Tmctil-sub-Gaussian-2}
			\max_{i=1,\dotsc,n} \sup_{t \ge 0} p(n,t,\varepsilon,i,1) \le Ce^{-K\sqrt{n\varepsilon}}, \quad \forall n \in \Nmb, \: \varepsilon>0.
		\end{equation}
		If in addition \eqref{eq:condition-special} holds, then $\sqrt{n\varepsilon}$ in \eqref{eq:Tmctil-sub-Gaussian-2} could be improved to $n\varepsilon$.
	\end{enumerate}	
\end{Lemma}

\begin{proof}
	We first prove part (ii).
	Fix $n,t,\varepsilon,i$.
	Applying Markov's inequality, we have
	\begin{align}
		p(n,t,\varepsilon,i,1) 
		& = \Pmb\left( \int_0^t e^{-\kappa  (t-s)} \left| \frac{1}{n} \sum_{j=1}^n \left( \xi_{ij}^n - G(\frac{i}{n},\frac{j}{n}) \right) b(\Xlimit_{\frac{i}{n}}(s),0) \right|^2 ds > \varepsilon/16 \right) \label{eq:pn1} \\
		& = \Pmb\left( \left| \frac{1}{n} \sum_{j=1}^n \left( \xi_{ij}^n - G(\frac{i}{n},\frac{j}{n}) \right) \right| \sqrt{\int_0^t e^{-\kappa  (t-s)} \left| b(\Xlimit_{\frac{i}{n}}(s),0) \right|^2 ds} > \sqrt{\varepsilon}/4 \right) \notag \\
		& \le e^{-\theta \sqrt{\varepsilon}/4} \Emb \left[ \exp \left( \theta \left| \frac{1}{n} \sum_{j=1}^n \left( \xi_{ij}^n - G(\frac{i}{n},\frac{j}{n}) \right) \right| \sqrt{\int_0^t e^{-\kappa  (t-s)} \left| b(\Xlimit_{\frac{i}{n}}(s),0) \right|^2 ds} \right) \right] \notag
	\end{align}
	for each $\theta > 0$.
	Since $e^{|x|} \le e^x + e^{-x}$, we have
	\begin{align*}
		& \Emb \left[ \exp \left( \theta \left| \frac{1}{n} \sum_{j=1}^n \left( \xi_{ij}^n - G(\frac{i}{n},\frac{j}{n}) \right) \right| \sqrt{\int_0^t e^{-\kappa  (t-s)} \left| b(\Xlimit_{\frac{i}{n}}(s),0) \right|^2 ds} \right) \right] \\
		& \le \Emb \left[ \exp \left( \frac{\theta}{n} \sum_{j=1}^n \left( \xi_{ij}^n - G(\frac{i}{n},\frac{j}{n}) \right) \sqrt{\int_0^t e^{-\kappa  (t-s)} \left| b(\Xlimit_{\frac{i}{n}}(s),0) \right|^2 ds} \right) \right] \\
		& \qquad + \Emb \left[ \exp \left( - \frac{\theta}{n} \sum_{j=1}^n \left( \xi_{ij}^n - G(\frac{i}{n},\frac{j}{n}) \right) \sqrt{\int_0^t e^{-\kappa  (t-s)} \left| b(\Xlimit_{\frac{i}{n}}(s),0) \right|^2 ds} \right) \right].
	\end{align*}	
	Since $0 \le \xi_{ij}^n \le 1$ and $\Emb \left[\xi_{ij}^n - G(\frac{i}{n},\frac{j}{n}) \,|\, \Xlimit_{\frac{i}{n}} \right] = 0$, we can condition on $\Xlimit_{i/n}$ and apply Hoeffding's lemma to obtain that
	\begin{align*}
		& \Emb \left[ \exp \left( \pm \frac{\theta}{n} \sum_{j=1}^n \left( \xi_{ij}^n - G(\frac{i}{n},\frac{j}{n}) \right) \sqrt{\int_0^t e^{-\kappa  (t-s)} \left| b(\Xlimit_{\frac{i}{n}}(s),0) \right|^2 ds} \right) \right] \\
		& = \Emb \left\{ \prod_{j=1}^n \Emb \left[ \exp \left( \pm \frac{\theta}{n} \left( \xi_{ij}^n - G(\frac{i}{n},\frac{j}{n}) \right) \sqrt{\int_0^t e^{-\kappa  (t-s)} \left| b(\Xlimit_{\frac{i}{n}}(s),0) \right|^2 ds} \right) \,\Big|\, \Xlimit_{\frac{i}{n}} \right] \right\}\\
		& \le \Emb \left[ \exp \left( \frac{\theta^2}{8n} \int_0^t e^{-\kappa  (t-s)} \left| b(\Xlimit_{\frac{i}{n}}(s),0) \right|^2 ds \right) \right] \\
		& \le \Emb \left[ \int_0^t C_1 e^{-\kappa  (t-s)} \exp \left( \frac{C_1\theta^2}{8n} \left| b(\Xlimit_{\frac{i}{n}}(s),0) \right|^2 \right) ds \right] \\
		& \le \Emb \left[ \int_0^t C_1 e^{-\kappa  (t-s)} \exp \left( \frac{C_2\theta^2}{8n} (1+|\Xlimit_{\frac{i}{n}}(s)|^2) \right) ds \right],
	\end{align*}
	where the second inequality follows from Jensen's inequality, and the last line uses the linear growth property of $b$.
	Combining these three estimates gives
	\begin{equation*}
		p(n,t,\varepsilon,i,1) \le 2e^{-\theta \sqrt{\varepsilon}/4} \Emb \left[ \int_0^t C_1 e^{-\kappa  (t-s)} \exp \left( \frac{C_2\theta^2}{8n} (1+|\Xlimit_{\frac{i}{n}}(s)|^2) \right) ds \right].
	\end{equation*} 
	Taking $\theta = \sqrt{n(\kappa \wedge \theta_0)/C_2}$ and using Lemma \ref{lem:square-exponential-moment}(ii), we have the desired result \eqref{eq:Tmctil-sub-Gaussian-2}.
	
	Now we prove the strengthened version of \eqref{eq:Tmctil-sub-Gaussian-2} under the extra assumption that \eqref{eq:condition-special} holds, namely $\xi_{ij}^n=G(\frac{i}{n},\frac{j}{n}) \mbox{ or } b(\cdot,0) \mbox{ is bounded}$.	
	If $\xi_{ij}^n=G(\frac{i}{n},\frac{j}{n})$, from \eqref{eq:pn1} we have $p(n,t,\varepsilon,i,1)=0$ and the desired result clearly holds.	
	If $b(\cdot,0)$ is bounded, from \eqref{eq:pn1} we have
	\begin{align*}
		p(n,t,\varepsilon,i,1) 
		& \le \Pmb\left( \left| \frac{1}{n} \sum_{j=1}^n \left( \xi_{ij}^n - G(\frac{i}{n},\frac{j}{n}) \right) \right|^2 > C_3\varepsilon \right) \\
		& \le e^{-\theta C_3 \varepsilon} \Emb \left[ \exp \left( \theta \left| \frac{1}{n} \sum_{j=1}^n \left( \xi_{ij}^n - G(\frac{i}{n},\frac{j}{n}) \right) \right|^2 \right) \right]
	\end{align*}	
	for each $\theta > 0$.
	Letting $Z$ be an independent standard normal random variable, we have
	\begin{align*}
		\Emb \left[ \exp \left( \theta \left| \frac{1}{n} \sum_{j=1}^n \left( \xi_{ij}^n - G(\frac{i}{n},\frac{j}{n}) \right) \right|^2 \right) \right] & = \Emb \Big[ \exp \Big(Z \frac{\sqrt{2\theta}}{n} \sum_{j=1}^n \left( \xi_{ij}^n - G(\frac{i}{n},\frac{j}{n}) \right) \Big) \Big] \\
		& = \Emb \Big\{ \prod_{j=1}^n \Emb \Big[ \exp \Big(Z \frac{\sqrt{2\theta}}{n} \left( \xi_{ij}^n - G(\frac{i}{n},\frac{j}{n}) \right) \Big) \,\Big|\, Z \Big] \Big\} \\
		& \le \Emb \left[ \exp \left(\frac{\theta Z^2}{4n}\right) \right].
	\end{align*}
	where the last line uses Hoeffding's lemma again.
	It then follows from the formula of the moment generating function of $Z^2$ that
	$$\Emb \left[ \exp \left( \theta \left| \frac{1}{n} \sum_{j=1}^n \left( \xi_{ij}^n - G(\frac{i}{n},\frac{j}{n}) \right) \right|^2 \right) \right] \le \left(1-\frac{\theta}{2n}\right)^{-1/2}.$$
	Taking $\theta = n$ gives the strengthened version of \eqref{eq:Tmctil-sub-Gaussian-2}.
	
	The proof of part (i) is the same as that of part (ii), except that Lemma \ref{lem:square-exponential-moment}(ii) and all the relevant constants (such as $C_1$) are to be replaced by Lemma \ref{lem:square-exponential-moment}(i) and $T$-dependent constants (such as $C_1(T)$).
\end{proof}

For the analysis of $p(n,t,\varepsilon,i,2)$, we will need the following standard property of sub-Gaussian random vectors. A proof is provided in Appendix \ref{sec:appendix-2} for completeness.

\begin{Lemma}
	\label{lem:sub-Gaussian}
	Let $Y$ be an $\Rd$-valued random variable with $\Emb Y=0$ and $\Emb[e^{|Y|^2/a}] \le 2$ for some $a \in (0,\infty)$. Then for any $\lambda \in \Rd$,
	\begin{equation*}
		\Emb[e^{\lambda \cdot Y}] \le \exp \left(\frac{5a}{2}|\lambda|^2\right).
	\end{equation*}
\end{Lemma}

Recall $\Tmc_t^{n,i}$ defined in \eqref{eq:Tmc}.
Let
\begin{align}
	\Tmc^{n,i,j}_t & := \left( \xi_{ij}^n - G(\frac{i}{n},\frac{j}{n}) \right) \left( b(\Xlimit_{\frac{i}{n}}(t),\Xlimit_{\frac{j}{n}}(t)) - b(\Xlimit_{\frac{i}{n}}(t),0) \right) \notag \\
	& \qquad + G(\frac{i}{n},\frac{j}{n}) \left( b(\Xlimit_{\frac{i}{n}}(t),\Xlimit_{\frac{j}{n}}(t)) - \int_{\Rmb^d} b(\Xlimit_{\frac{i}{n}}(t),x) \,\lawlimit_{\frac{j}{n},t}(dx) \right). \label{eq:Tmc-ij}
\end{align}	
Note that
$\Tmc_t^{n,i}=\frac{1}{n}\sum_{j=1}^n \Tmc^{n,i,j}_t$.
We will estimate $p(n,t,\varepsilon,i,2)$ by proving the following sub-Gaussian bounds of $\Tmc^{n,i,j}_t$.

\begin{Lemma}
	\phantomsection
	\label{lem:Tmc}
	\begin{enumerate}[(i)]
	\item 
		For every $T \in (0,\infty)$, there exists some $C(T) \in (0,\infty)$ such that
		\begin{equation*}
			\max_{i=1,\dotsc,n} \sup_{t \in [0,T]} \Emb \Big[ \exp \Big(s\Big|\frac{1}{n}\sum_{j \ne i} \Tmc^{n,i,j}_t\Big|^2\Big) \Big] \le \left( 1-\frac{C(T)s}{n} \right)^{-d/2}, \: \forall \, s \in \left(0,\frac{n}{C(T)}\right), n \in \Nmb.
		\end{equation*}	
	\item 
		Suppose Condition \ref{cond:strong} holds. 
		Then there exists some $C \in (0,\infty)$ such that
		\begin{equation*}
			\max_{i=1,\dotsc,n} \sup_{t \ge 0} \Emb \Big[ \exp \Big(s\Big|\frac{1}{n}\sum_{j \ne i} \Tmc^{n,i,j}_t\Big|^2\Big) \Big] \le \left( 1-\frac{Cs}{n} \right)^{-d/2}, \: \forall \, s \in \left(0,\frac{n}{C}\right), n \in \Nmb.
		\end{equation*}
	\end{enumerate}	
\end{Lemma}

\begin{proof}
	We first prove part (ii).
	Fix $t,n,i$ and $j \ne i$.
	Using the Lipschitz property of $b$ and Lemma \ref{lem:square-exponential-moment}(ii), we have
	$$|\Tmc^{n,i,j}_t|^2 \le C_1(1+|\Xlimit_{\frac{j}{n}}(t)|^2).$$
	It then follows from Lemma \ref{lem:square-exponential-moment}(ii) that
	\begin{equation*}
		\Emb\left[ \exp \left(|\Tmc^{n,i,j}_t|^2/C_2\right) \,\big|\, \Xlimit_{\frac{i}{n}}(t)\right] \le 2.
	\end{equation*}
	Using this and the fact that $\Emb\left[ \Tmc^{n,i,j}_t \,\big|\, \Xlimit_{\frac{i}{n}}(t)\right] = 0$, we can apply Lemma \ref{lem:sub-Gaussian} to get
	\begin{equation*}
		\Emb\left[ \exp(\lambda \cdot \Tmc^{n,i,j}_t) \,\big|\, \Xlimit_{\frac{i}{n}}(t)\right] \le \exp \left(\frac{5C_2}{2}|\lambda|^2\right)
	\end{equation*}
	for each $\lambda \in \Rmb^d$.
	Let $Z$ be an independent $d$-dimensional standard normal random vector.
	Using the independence of $(Z, \{\xi_{ij}^n\}, \{\Xlimit_{\frac{i}{n}}(t)\})$, we have that for any $s>0$,
	\begin{align*}
		\Emb \Big[ \exp \Big(s\Big|\frac{1}{n}\sum_{j \ne i} \Tmc^{n,i,j}_t\Big|^2\Big) \Big] & = \Emb \Big[ \exp \Big(Z \cdot \frac{\sqrt{2s}}{n}\sum_{j \ne i} \Tmc^{n,i,j}_t \Big) \Big] \\
		& = \Emb \Big\{ \prod_{j \ne i} \Emb \Big[ \exp \Big(Z \cdot \frac{\sqrt{2s}}{n} \Tmc^{n,i,j}_t \Big) \,\Big|\, Z, \Xlimit_{\frac{i}{n}}(t) \Big] \Big\} \\
		& \le \Emb \left[ \exp \left(\frac{5C_2}{2}\left|\frac{\sqrt{2s}Z}{n}\right|^2(n-1)\right) \right].
	\end{align*}
	It then follows from the formula of the moment generating function of $|Z|^2$ that
	\begin{equation*}
		\Emb \Big[ \exp \Big(s\Big|\frac{1}{n}\sum_{j \ne i} \Tmc^{n,i,j}_t\Big|^2\Big) \Big] \le \left( 1-\frac{C_3s}{n} \right)^{-d/2},
	\end{equation*}
	for all $t \ge 0$, $n \in \Nmb$, $i=1,\dotsc,n$, $s < n/C_3$.
	This gives part (ii).
	
	The proof of part (i) is the same as that of part (ii), except that Lemma \ref{lem:square-exponential-moment}(ii) and all the relevant constants (such as $C_1$) are to be replaced by Lemma \ref{lem:square-exponential-moment}(i) and $T$-dependent constants (such as $C_1(T)$).
\end{proof}

Now we show the following estimate for $p(n,t,\varepsilon,i,2)$.
Recall that
$$p(n,t,\varepsilon,i,2) = \Pmb\left( \int_0^t e^{-\kappa  (t-s)} |\Tmc_s^{n,i}|^2 ds > \varepsilon/16 \right),$$
where $\Tmc_s^{n,i}$ was introduced in \eqref{eq:Tmc}.

\begin{Lemma}
	\phantomsection
	\label{lem:pn2}
	\begin{enumerate}[(i)]
	\item 
		For every $T \in (0,\infty)$, there exist some $C(T), K_T \in (0,\infty)$ such that
		\begin{equation*}
			\max_{i=1,\dotsc,n} \sup_{t \in [0,T]} p(n,t,\varepsilon,i,2) \le C(T)e^{-K_Tn\varepsilon}, \quad \forall n \in \Nmb, \: \varepsilon>0.
		\end{equation*}	
	\item
		Suppose Condition \ref{cond:strong} holds.
		Then there exist some $C,K \in (0,\infty)$ such that
		\begin{equation*}
			\max_{i=1,\dotsc,n} \sup_{t \ge 0} p(n,t,\varepsilon,i,2) \le Ce^{-Kn \varepsilon}, \quad \forall n \in \Nmb, \: \varepsilon>0.
		\end{equation*}	
	\end{enumerate}
\end{Lemma}

\begin{proof}
	We first prove part (ii).
	Fix $n,t,\varepsilon,i$.
	Applying Markov's inequality and Jensen's inequality, we have
	\begin{align*}
		p(n,t,\varepsilon,i,2)
		& \le e^{-\theta \varepsilon/16} \Emb \left[\exp \left( \theta \int_0^t e^{-\kappa  (t-s)} |\Tmc_s^{n,i}|^2\, ds \right) \right] \\
		& \le e^{-\theta \varepsilon/16} \Emb \left[\int_0^t C_1e^{-\kappa  (t-s)} \exp \left( C_1\theta |\Tmc_s^{n,i}|^2 \right) ds \right]
	\end{align*}
	for each $\theta > 0$.
	Recall $\Tmc^{n,i,j}_t$ in \eqref{eq:Tmc-ij}.
	Using Holder's inequality and Lemma \ref{lem:Tmc}(ii), we have
	\begin{align*}
		\Emb \left[ \exp \left( C_1\theta |\Tmc_t^{n,i}|^2 \right) \right] 
		& = \Emb \left[ \exp \left( C_1\theta \left| \frac{1}{n} \sum_{j=1}^n \Tmc^{n,i,j}_t \right|^2 \right) \right] \\
		& \le \Emb \left[ \exp \left( 2C_1\theta \left| \frac{1}{n} \Tmc^{n,i,i}_t \right|^2 + 2C_1\theta \left| \frac{1}{n} \sum_{j \ne i}^n \Tmc^{n,i,j}_t \right|^2 \right) \right] \\
		& \le \sqrt{\Emb \left[ \exp \left( 4C_1\theta \left| \frac{1}{n} \Tmc^{n,i,i}_t \right|^2 \right) \right] \Emb \left[ \exp \left( 4C_1\theta \left| \frac{1}{n} \sum_{j \ne i}^n \Tmc^{n,i,j}_t \right|^2 \right) \right]} \\
		& \le \sqrt{\Emb \left[ \exp \left( 4C_1\theta \left| \frac{1}{n} \Tmc^{n,i,i}_t \right|^2 \right) \right] \left( 1-\frac{C_2\theta}{n} \right)^{-d/2}}.
	\end{align*}
	Now taking $\theta=\delta n$ for small enough $\delta>0$ and using Lemma \ref{lem:square-exponential-moment}(ii), we have
	\begin{equation*}
		p(n,t,\varepsilon,i,2) \le C_3e^{-C_4n\varepsilon}.
	\end{equation*}
	This gives part (ii).
	
	The proof of part (i) is the same as that of part (ii), except that Lemmas \ref{lem:square-exponential-moment}(ii), \ref{lem:Tmc}(ii) and all the relevant constants (such as $C_1$) are to be replaced by Lemmas \ref{lem:square-exponential-moment}(i), \ref{lem:Tmc}(i) and $T$-dependent constants (such as $C_1(T)$).
\end{proof}

\subsection{Proofs of Theorems \ref{thm:finite} and \ref{thm:infinite}}
\label{sec:complete-pf}

Finally we need the following result on the concentration estimates of the distance $W_1(\emplimit^n_t, \lawlimitmean_t)$ in Proposition \ref{prop:prep}.

\begin{Lemma}
	\phantomsection
	\label{lem:concentration-empirical}
	\begin{enumerate}[(i)]
	\item 
		For every $T \in (0,\infty)$, there exists some $K_T \in (0,\infty)$ such that for any $d' > d$, there exists some $N_T \in \Nmb$ such that
		\begin{equation*}
			\sup_{0 \le t \le T} \Pmb\left( W_1(\emplimit^n_t,\lawlimitmean_t) > \varepsilon \right) \le \exp \left(-K_T n\varepsilon^2 \right),
		\end{equation*}		
		for	all $\varepsilon>0$ and $n \ge N_T \max(\varepsilon^{-(d'+2)},1)$.
	\item
		Suppose Condition \ref{cond:strong} holds.
		Then there exists some $K \in (0,\infty)$ such that for any $d' > d$, there exists some $N_0 \in \Nmb$ such that
		\begin{equation*}
			\sup_{t \ge 0} \Pmb\left( W_1(\emplimit^n_t,\lawlimitmean_t) > \varepsilon \right) \le \exp \left(-K n\varepsilon^2 \right),
		\end{equation*}		
		for	all $\varepsilon>0$ and $n \ge N_0 \max(\varepsilon^{-(d'+2)},1)$.
	\end{enumerate}
\end{Lemma}

\begin{proof}
	We first prove part (ii).
	Let $\lawlimit^n_t := \frac{1}{n} \sum_{i=1}^n \lawlimit_{\frac{i}{n},t}$.
	From Lemma \ref{lem:square-exponential-moment}(ii) we know that 
	$$\sup_{t \ge 0} \sup_{n \in \Nmb} \max_{i=1,\dotsc,n} \int_{\Rd} e^{\alpha |x|^2} \, \lawlimit_{\frac{i}{n},t}(dx) < \infty$$ 
	for some $\alpha > 0$.
	Then there is some $\lambda>0$ such that $\lawlimit^n_t$ satisfies the Talagrand (transport) $T_1(\lambda)$ inequality (see \eqref{eq:Talagrand} in Proposition \ref{prop:concentration-empirical}) for each $n \in \Nmb$ and $t \ge 0$ (by e.g.\ \cite[Corollary 2.4]{BolleyVillani2005weighted}).
	It then follows from Proposition \ref{prop:concentration-empirical} that for any $d' > d$, there exists some $N_0 \in \Nmb$ such that
	\begin{equation*}
		\sup_{t \ge 0} \Pmb\left( W_1(\emplimit^n_t,\lawlimit^n_t) > \varepsilon/2 \right) \le \exp \left(-4\lambda n\varepsilon^2 \right),
	\end{equation*}		
	for	all $\varepsilon>0$ and $n \ge N_0 \max(\varepsilon^{-(d'+2)},1)$.
	From Lemma \ref{lem:Lipschitz-special} we know that
	\begin{equation*}
		\sup_{t \ge 0} W_1(\lawlimit^n_t,\lawlimitmean_t) \le \varepsilon/2
	\end{equation*}
	for all $n \ge C_1/\varepsilon$.
	Combining these two estimates with the triangle inequality gives the desired result.
	
	The proof of part (i) is the same as that of part (ii), except that Lemmas \ref{lem:Lipschitz-special}, \ref{lem:square-exponential-moment}(ii) and all the relevant constants (such as $C_1$) are to be replaced by Lemmas \ref{lem:Lipschitz}, \ref{lem:square-exponential-moment}(i) and $T$-dependent constants (such as $C_1(T)$). 
\end{proof}

Now we are ready to prove Theorem \ref{thm:infinite}.

\begin{proof}[Proof of Theorem \ref{thm:infinite}]

Using Proposition \ref{prop:prep}(ii) and recalling \eqref{eq:pn}, we have
\begin{equation*}
	\Pmb(W_1(\emppre^n_t, \lawlimitmean_t) > \varepsilon) \le \Pmb(W_1(\emplimit^n_t, \lawlimitmean_t) > \varepsilon/2) + p(n,t,C_1\varepsilon^2).
\end{equation*}
The result then follows on combining this with Lemma \ref{lem:concentration-empirical}(ii) and Proposition \ref{prop:p-n-eps}(ii).
\end{proof}

Finally we prove Theorem \ref{thm:finite}.

\begin{proof}[Proof of Theorem \ref{thm:finite}]
Fix $T,\varepsilon \in (0,\infty)$ and $d'>d$.
We will suppress the dependency of constants $C(T)$'s on $T$.
From Proposition \ref{prop:prep}(i) we have
\begin{align*}
	W_1(\emppre^n_t, \lawlimitmean_t) & \le W_1(\emplimit^n_t, \lawlimitmean_t) + C_1 \left( \int_0^t e^{-\kappa  (t-s)} \frac{1}{n} \sum_{i=1}^n \left| |\Rmc_s^{n,i}| + \frac{1}{n} \left( 1+|\Xlimit_{\frac{i}{n}}(s)| \right) \right|^2 ds \right)^{1/2} \\
	& \le W_1(\emplimit^n_t, \lawlimitmean_t) + C_1 e^{|\kappa|(T-t)/2} \left( \int_0^T e^{-\kappa  (T-s)} \frac{1}{n} \sum_{i=1}^n \left| |\Rmc_s^{n,i}| + \frac{1}{n} \left( 1+|\Xlimit_{\frac{i}{n}}(s)| \right) \right|^2 ds \right)^{1/2}.
\end{align*}
Therefore
\begin{align}
	& \Pmb \left(\sup_{0 \le t \le T} W_1(\emppre^n_t, \lawlimitmean_t) > \varepsilon\right) \notag \\ 
	& \le \Pmb \left(\sup_{0 \le t \le T} W_1(\emplimit^n_t, \lawlimitmean_t) > \frac{\varepsilon}{2}\right) + \Pmb \left(\int_0^T e^{-\kappa  (T-s)} \frac{1}{n} \sum_{i=1}^n \left| |\Rmc_s^{n,i}| + \frac{1}{n} \left( 1+|\Xlimit_{\frac{i}{n}}(s)| \right) \right|^2 ds > C_2\varepsilon^2\right) \notag \\
	& =: \ptil(n,T,\varepsilon,1) + \ptil(n,T,\varepsilon,2). \label{eq:finite-time-1}
\end{align}

We first analyze $\ptil(n,T,\varepsilon,1)$.
Let $\Delta > 0$ (to be chosen later) and $M := \lceil \frac{T}{\Delta} \rceil$.
Decompose the interval $[0,T]$ as
$$[0,T] = [0,\Delta] \cup [\Delta,2\Delta] \cup \dotsb \cup [(M-1)\Delta,T] \subset \cup_{h=0}^{M-1} [h\Delta,(h+1)\Delta].$$
Then, by triangle inequality, we have
\begin{align}
	\ptil(n,T,\varepsilon,1) & \le \Pmb \left(\sup_{h=0,\dotsc,M-1} \sup_{h\Delta \le t \le (h+1)\Delta} W_1(\emplimit^n_t, \lawlimitmean_t) > \frac{\varepsilon}{2}\right) \notag \\
	& \le \Pmb \left(\sup_{h=0,\dotsc,M-1} \sup_{h\Delta \le t \le (h+1)\Delta} W_1(\emplimit^n_t, \emplimit^n_{h\Delta}) > \frac{\varepsilon}{6}\right) \notag \\
	& \qquad + \Pmb \left(\sup_{h=0,\dotsc,M-1} W_1(\emplimit^n_{h\Delta}, \lawlimitmean_{h\Delta}) > \frac{\varepsilon}{6}\right) \notag \\
	& \qquad + \Pmb \left(\sup_{h=0,\dotsc,M-1} \sup_{h\Delta \le t \le (h+1)\Delta} W_1(\lawlimitmean_{h\Delta}, \lawlimitmean_t) > \frac{\varepsilon}{6}\right). \label{eq:finite-time-2}
\end{align}

For the first term on the right hand side of \eqref{eq:finite-time-2}, from Lemma \ref{lem:regularity-n} we have
$$\Pmb \left(\sup_{h\Delta \le t \le (h+1)\Delta} W_1(\emplimit^n_t, \emplimit^n_{h\Delta}) > \frac{\varepsilon}{6}\right) \le \exp \left( -n \left( C_3\varepsilon^2-C_4\Delta\right) \right)$$
for all $h=0,\dotsc,M-1$.
So
$$\Pmb \left(\sup_{h=0,\dotsc,M-1} \sup_{h\Delta \le t \le (h+1)\Delta} W_1(\emplimit^n_t, \emplimit^n_{h\Delta}) > \frac{\varepsilon}{6}\right) \le M \exp \left( -n \left( C_3\varepsilon^2-C_4\Delta\right) \right).$$
Taking $\Delta \le C_3\varepsilon^2/2C_4$, we have $M \le C_5(T/\varepsilon^2 + 1)$ and
\begin{equation}
	\label{eq:finite-time-3}
	\Pmb \left(\sup_{h=0,\dotsc,M-1} \sup_{h\Delta \le t \le (h+1)\Delta} W_1(\emplimit^n_t, \emplimit^n_{h\Delta}) > \frac{\varepsilon}{6}\right) \le C_5\left( \frac{T}{\varepsilon^2} + 1 \right) \exp \left( -\frac{C_3}{2}n\varepsilon^2 \right).
\end{equation}
For the second term on the right hand side of \eqref{eq:finite-time-2}, by Lemma \ref{lem:concentration-empirical}(i), there exist some $K \in (0,\infty)$ and $N_0 \in \Nmb$ such that 
\begin{equation*}
	\Pmb(W_1(\emplimit^n_{h\Delta}, \lawlimitmean_{h\Delta}) > \varepsilon/6) \le e^{-Kn\varepsilon^2}
\end{equation*}
for all $\varepsilon>0$ and $n \ge N_0 \max(\varepsilon^{-(d'+2)},1)$.
Hence
\begin{align}
	\Pmb\left(\sup_{h=0,\dotsc,M-1} W_1(\emplimit^n_{h\Delta}, \lawlimitmean_{h\Delta}) > \frac{\varepsilon}{6}\right) & \le \sum_{h=0}^{M-1} \Pmb(W_1(\emplimit^n_{h\Delta}, \lawlimitmean_{h\Delta}) > \varepsilon/6) \notag \\
	& \le M e^{-Kn\varepsilon^2} \le C_5\left( \frac{T}{\varepsilon^2} + 1 \right) \exp \left( -Kn \varepsilon^2 \right). \label{eq:finite-time-4}
\end{align}
For the last term on the right hand side of \eqref{eq:finite-time-2}, taking $\Delta \le C_6 \varepsilon^2$ for some $C_6 \in (0,\infty)$ small enough, from Lemma \ref{lem:regularity} we have
\begin{equation*}
	h\Delta \le t \le (h+1)\Delta \Rightarrow W_1(\lawlimitmean_t, \lawlimitmean_{h\Delta}) \le \varepsilon/6,
\end{equation*}
and hence
\begin{equation}
	\label{eq:finite-time-5}
	\Pmb \left(\sup_{h=0,\dotsc,M-1} \sup_{h\Delta \le t \le (h+1)\Delta} W_1(\lawlimitmean_{h\Delta}, \lawlimitmean_t) > \frac{\varepsilon}{6}\right)=0. 
\end{equation}
Combining \eqref{eq:finite-time-2}--\eqref{eq:finite-time-5} and taking $\Delta = \min(C_3\varepsilon^2/2C_4, C_6 \varepsilon^2)$, we have
\begin{equation}
	\label{eq:finite-time-6}
	\ptil(n,T,\varepsilon,1) \le 2C_5\left( \frac{T}{\varepsilon^2} + 1 \right) \exp \left( -C_7n\varepsilon^2 \right)
\end{equation}
for all $\varepsilon>0$ and $n \ge N_0 \max(\varepsilon^{-(d'+2)},1)$.

Next we analyze $\ptil(n,T,\varepsilon,2)$ in \eqref{eq:finite-time-1}.
Using \eqref{eq:pn} and Proposition \ref{prop:p-n-eps}(i) we have
\begin{equation}
	\label{eq:finite-time-7}
	\ptil(n,T,\varepsilon,2) = p(n,T,C_2\varepsilon^2) \le C_8e^{-C_9\sqrt{n} \varepsilon},
\end{equation}
for all $\varepsilon > 0$ and $n \ge C_{10}/\varepsilon$.

Combining \eqref{eq:finite-time-1}, \eqref{eq:finite-time-6} and \eqref{eq:finite-time-7}, we have the desired result
\begin{equation}
	\label{eq:finite-time-8}
	\Pmb \left(\sup_{0 \le t \le T} W_1(\emppre^n_t, \lawlimitmean_t) > \varepsilon\right) \le C_{11} \left( \frac{T}{\varepsilon^2} + 1 \right) e^{- C_{12}\sqrt{n} \varepsilon}
\end{equation}
for all $\varepsilon > 0$ and $n \ge \max(N_0,C_{10}) \max(\varepsilon^{-(d'+2)},1)$.
If in addition \eqref{eq:condition-special} holds, then using Proposition \ref{prop:p-n-eps}(i) we can improve $\sqrt{n}\varepsilon$ in \eqref{eq:finite-time-7} and \eqref{eq:finite-time-8} to $n\varepsilon^2$.
This completes the proof.
\end{proof}

\section{Acknowledgment}
We would like to thank the anonymous referees for their careful reading and many valuable suggestions on the paper.

\appendix

\section{Proofs of results in Section \ref{sec:property}}
\label{sec:appendix-1}

\begin{proof}[Proof of Lemma \ref{lem:square-exponential-moment}]
	The proof is similar to that of \cite[Proposition 4.3]{BudhirajaFan2017} and hence we only provide a sketch.
	
	(i) For each $t \in [0,T]$ and $u \in I$, using Proposition \ref{prop:well-posedness}(i) and the linear growth properties of $f$ and $b$, we have
	\begin{equation*}
		|\Xlimit_u(t)| \le |\Xlimit_u(0)| + \int_0^t C_1(T)(1+|\Xlimit_u(s)|)\,ds + |B_u(t)|.
	\end{equation*}
	It then follows from Gronwall's inequality that
	\begin{equation*}
		\sup_{t \in [0,T]} |\Xlimit_u(t)| \le C_2(T)\left( |\Xlimit_u(0)| + 1 + \sup_{t \in [0,T]} |B_u(t)| \right).
	\end{equation*}
	Thus for $\theta > 0$,
	\begin{align*}
		\Emb\left[ \sup_{t \in [0,T]} e^{\theta |\Xlimit_u(t)|^2}\right] & \le \Emb\left[ e^{C_3(T)\theta |\Xlimit_u(0)|^2} e^{C_3(T)\theta} e^{C_3(T)\theta \sup_{t \in [0,T]} |B_u(t)|^2} \right] \\
		& \le \left( \Emb\left[ e^{2C_3(T)\theta |\Xlimit_u(0)|^2} \right] \right)^{1/2} e^{C_3(T)\theta} \left( \Emb\left[e^{2C_3(T)\theta \sup_{t \in [0,T]} |B_u(t)|^2} \right] \right)^{1/2}.
	\end{align*}
	The proof of part (i) follows by choosing $\theta>0$ such that $\Emb\left[ e^{2C_3(T)\theta |\Xlimit_u(0)|^2} \right]<\infty$ and $\Emb\left[e^{2C_3(T)\theta \sup_{t \in [0,T]} |B_u(t)|^2} \right]<\infty$.
	
	(ii) Fix $\theta \in (0,(\kappa \wedge \theta_0)/4)$. 
	Note that for the function $\phi(x):=e^{\theta|x|^2}$, we have $\nabla\phi(x)=2\theta e^{\theta|x|^2}x$ and $\Delta\phi(x)=2\theta e^{\theta|x|^2}(d+2\theta|x|^2)$.
	It then follows from It\^o's formula that
	\begin{align}
		d e^{\theta|\Xlimit_u(t)|^2} & = e^{\theta|\Xlimit_u(t)|^2} \left[2\theta\Xlimit_u(t) \cdot \left( f(\Xlimit_u(t))\,dt + \int_0^1 \int_{\Rmb^d} b(\Xlimit_u(t),x)G(u,v) \,\lawlimit_{v,t}(dx)\,dv\,dt \right. \right. \notag \\
		& \quad + dB_u(t) \Big) + \theta \left( d+2\theta|\Xlimit_u(t)|^2 \right)dt \Big]. \label{eq:square-exp-pf-1}
	\end{align}
	Using \eqref{eq:c0} we have
	\begin{align}
		\Xlimit_u(t) \cdot f(\Xlimit_u(t)) & = \left(\Xlimit_u(t)-0\right) \cdot \left(f(\Xlimit_u(t))-f(0)\right) + \Xlimit_u(t) \cdot f(0) \notag \\
		& \le -c_0|\Xlimit_u(t)|^2 + C_1|\Xlimit_u(t)| \le -(c_0- \half K_b)|\Xlimit_u(t)|^2 + C_2. \label{eq:square-exp-pf-2}
	\end{align}
	Using Lemma \ref{lem:Lipschitz-special} and the Lipschitz property of $b$ we have
	\begin{equation}
		\label{eq:square-exp-pf-3}
		\Xlimit_u(t) \cdot \int_0^1 \int_{\Rmb^d} b(\Xlimit_u(t),x)G(u,v) \,\lawlimit_{v,t}(dx)\,dv \le K_b|\Xlimit_u(t)|^2 + C_3|\Xlimit_u(t)| \le \frac{3}{2}K_b|\Xlimit_u(t)|^2 + C_4.
	\end{equation}	
	Similar to the proof of \cite[Proposition 4.3(ii)]{BudhirajaFan2017}, one can verify that
	$$\int_0^t e^{\theta|\Xlimit_u(s)|^2} \Xlimit_u(s) \cdot dB_u(s)$$
	is a square integrable martingale for all $\theta \in (0,(\kappa \wedge \theta_0)/4)$.
	Then 
	$$\Emb \int_{t_1}^{t_2} e^{\theta|\Xlimit_u(t)|^2} 2\theta\Xlimit_u(t) \cdot dB_u(t)=0, \quad \forall\, t_2 \ge t_1 \ge 0.$$
	Combining this and \eqref{eq:square-exp-pf-1}--\eqref{eq:square-exp-pf-3} gives
	\begin{equation*}
		d\Emb \left[e^{\theta|\Xlimit_u(t)|^2}\right] \le \Emb \left[ e^{\theta|\Xlimit_u(t)|^2} \left(A + B|\Xlimit_u(t)|^2\right) \right] dt,
	\end{equation*}
	where
	\begin{equation*}
		A = \theta C_5, \quad
		B = 2\theta^2-2\kappa\theta.
	\end{equation*}
	Since $\theta < (\kappa \wedge \theta_0)/4 < \kappa$, we have $B<0$.
	Therefore
	\begin{equation}
		\label{eq:square-exp-pf-4}
		d\Emb \left[e^{\theta|\Xlimit_u(t)|^2}\right] \le C_6(\theta) \Emb \left[ e^{\theta|\Xlimit_u(t)|^2} \left(C_7(\theta) - |\Xlimit_u(t)|^2\right) \right] dt.
	\end{equation}
	Decomposing the expectation on the right hand side of \eqref{eq:square-exp-pf-4} according to the size of $|\Xlimit_u(t)|$, we get
	\begin{equation*}
		d\Emb \left[e^{\theta|\Xlimit_u(t)|^2}\right] \le \left(C_8(\theta)-C_9(\theta)\Emb \left[e^{\theta|\Xlimit_u(t)|^2}\right]\right)dt.
	\end{equation*}
	Since $\theta < (\kappa \wedge \theta_0)/4 < \theta_0$, $\Emb \left[e^{\theta|\Xlimit_u(0)|^2}\right]<\infty$.
	A standard estimate shows that $\sup_{t \ge 0} \Emb \left[e^{\theta|\Xlimit_u(t)|^2}\right]<\infty$, uniformly in $u \in I$.
\end{proof}

\begin{proof}[Proof of Lemma \ref{lem:regularity-moment}]
	Fix $T \in (0,\infty)$, $0 \le s \le t \le T$ and $u \in I$.
	We will suppress the dependency of constants $C(T)$'s on $T$.
	From \eqref{eq:system} we have
	\begin{equation*}
		\Xlimit_u(t)-\Xlimit_u(s) = \int_s^t \left( f(\Xlimit_u(r)) + \int_0^1 \int_{\Rmb^d} b(\Xlimit_u(r),x)G(u,v) \,\lawlimit_{v,r}(dx)\,dv\right)dr + B_u(t) - B_u(s).
	\end{equation*}
	It then follows from the linear growth properties of $f$ and $b$ and Proposition \ref{prop:well-posedness}(i) that
	\begin{equation*}
		|\Xlimit_u(t)-\Xlimit_u(s)| \le C_1 \int_s^t \left( 1 + |\Xlimit_u(r)| \right)dr + |B_u(t) - B_u(s)|.
	\end{equation*}
	Using Proposition \ref{prop:well-posedness}(i) and Holder's inequality we have
	\begin{align*}
		\Emb |\Xlimit_u(t)-\Xlimit_u(s)|^4 & \le C_2 (t-s)^{3} \int_s^t \left( 1 + \Emb |\Xlimit_u(r)|^4 \right) dr + C_2 \Emb |B_u(t) - B_u(s)|^4 \\
		& \le C_3 (t-s)^4 + C_3 (t-s)^{2}.
	\end{align*}
	This gives the first assertion.
	
	For the last assertion, fix $t_0,\Delta \in [0,T]$, $t_0 \le s \le t \le t_0+\Delta$ and $u \in I$.
	Since
	\begin{align*}
		& \Emb \left[ \sup_{t_0 \le s \le t \le t_0+\Delta} \exp \left(\theta |\Xlimit_u(t)-\Xlimit_u(s)|^2 \right) \right] \\
		& \le \Emb \left[ \sup_{t_0 \le s \le t \le t_0+\Delta} \exp \left(2\theta |\Xlimit_u(t)-\Xlimit_u(t_0)|^2 + 2\theta |\Xlimit_u(s)-\Xlimit_u(t_0)|^2 \right) \right] \\
		& \le \Emb \left[ \sup_{t_0 \le t \le t_0+\Delta} \exp \left(4\theta |\Xlimit_u(t)-\Xlimit_u(t_0)|^2 \right) \right]
	\end{align*}
	for $\theta > 0$, it suffices to show that
	\begin{equation}
		\label{eq:regularity-moment-claim}
		\Emb \left[ \sup_{t_0 \le t \le t_0+\Delta} \exp \left(\theta_T |\Xlimit_u(t)-\Xlimit_u(t_0)|^2 \right) \right] \le 1+C(T)\Delta
	\end{equation}
	for some $\theta_T,C(T)>0$.
	Let $\theta \colon [0,T] \to \Rmb$ be a non-negative continuously differentiable function (to be chosen later) and write
	$$Z_u(r) := e^{\theta(r)|\Xlimit_u(r)-\Xlimit_u(t_0)|^2}.$$
	Using It\^o's formula we have
	\begin{align*}
		& Z_u(t) - 1 \\
		& = M_u(t) + \int_{t_0}^t Z_u(r) \left[ 2\theta(r)(\Xlimit_u(r)-\Xlimit_u(t_0)) \cdot \left( f(\Xlimit_u(r)) + \int_0^1 \int_{\Rmb^d} b(\Xlimit_u(r),x)G(u,v) \,\lawlimit_{v,r}(dx)\,dv\right) \right. \\
		& \left. \qquad \qquad + \theta(r) \left(d+2\theta(r)|\Xlimit_u(r)-\Xlimit_u(t_0)|^2\right) + \theta'(r)|\Xlimit_u(r)-\Xlimit_u(t_0)|^2\right]dr,
	\end{align*}
	where
	$$M_u(t):= \int_{t_0}^t 2Z_u(r) \theta(r)(\Xlimit_u(r)-\Xlimit_u(t_0)) \cdot dB_u(r).$$
	Using Lemma \ref{lem:square-exponential-moment}(i) we can choose some small $\zeta \in (0,\infty)$ such that 
	\begin{equation}
		\label{eq:regularity-moment-pf-1}
		\sup_{u \in I} \sup_{0 \le r \le T} \Emb e^{\zeta|\Xlimit_u(r)-\Xlimit_u(t_0)|^2} < \infty.
	\end{equation}
	The function $\theta(r)$ will be chosen below such that $\sup_{0 \le r \le T} \theta(r) < \zeta/2$.
	With such a choice of $\theta(r)$, $M_u(r)$ is a martingale.
	By Young's inequality and the linear growth properties of $f$ and $b$, we have that for every $\eta > 0$,
	\begin{align*}
		& 2(\Xlimit_u(r)-\Xlimit_u(t_0)) \cdot \left( f(\Xlimit_u(r)) + \int_0^1 \int_{\Rmb^d} b(\Xlimit_u(r),x)G(u,v) \,\lawlimit_{v,r}(dx)\,dv\right) \\
		& \le \eta |\Xlimit_u(r)-\Xlimit_u(t_0)|^2 + \frac{C_4(1+|\Xlimit_u(r)|^2)}{\eta}.
	\end{align*}
	So, by letting
	\begin{align*}
		A(r) & := \theta(r) \left[ d + \frac{C_4(1+|\Xlimit_u(r)|^2)}{\eta} \right], \\
		B(r) & := \eta \theta(r) + 2\theta^2(r) + \theta'(r),
	\end{align*}
	we obtain
	\begin{align*}
		Z_u(t) \le 1 + M_u(t) + \int_{t_0}^t Z_u(r) \left[ A(r) + B(r) |\Xlimit_u(r)-\Xlimit_u(t_0)|^2 \right] dr.
	\end{align*}
	The remaining argument is similar to that in \cite[Section 5.1]{BolleyGuillinVillani2007quantitative} and so we only give a sketch.
	Choose $\theta(r)$ to be the solution of the ODE
	$$\eta \theta(r) + 2\theta^2(r) + \theta'(r) = 0$$
	with $\theta(0) \in (0,\zeta/2)$.
	It is easy to see that the solution is decreasing and strictly positive.
	Thus $B(r)=0$ and $0 < \theta(T) \le \theta(r) < \zeta/2$ for every $r \in [0,T]$.
	As a consequence,
	$$\Emb \sup_{t_0 \le t \le t_0+\Delta} Z_u(t) \le 1 + \Emb \sup_{t_0 \le t \le t_0+\Delta} M_u(t) + \int_{t_0}^{t_0+\Delta} \Emb Z_u(r)A(r) \,dr.$$
	Using the bound in \eqref{eq:regularity-moment-pf-1} one can check exactly as in \cite[Section 5.1]{BolleyGuillinVillani2007quantitative} that
	$$\sup_{0 \le r \le T} \Emb Z_u(r)A(r) < \infty \mbox{ and } \Emb \sup_{t_0 \le t \le t_0+\Delta} M_u(t) \le C_5\Delta.$$
	This proves \eqref{eq:regularity-moment-claim}, with $C(T):=C_5+\sup_{0 \le r \le T} \Emb Z_u(r)A(r) < \infty$, and hence the result follows.
\end{proof}

\begin{proof}[Proof of Lemma \ref{lem:regularity-n}]
	Noting that 
	$W_1(\emplimit^n_s, \emplimit^n_t) \le \frac{1}{n} \sum_{i=1}^n |\Xlimit_{\frac{i}{n}}(t)-\Xlimit_{\frac{i}{n}}(s)|,$
	we have
	$$\Pmb \left(\sup_{t_0 \le s,t \le t_0+\Delta} W_1(\emplimit^n_s, \emplimit^n_t) > \varepsilon\right) \le \Pmb \left(\frac{1}{n} \sum_{i=1}^n V_i^n > \varepsilon\right),$$
	where
	$V_i^n := \sup_{t_0 \le s,t \le t_0+\Delta} |\Xlimit_{\frac{i}{n}}(t)-\Xlimit_{\frac{i}{n}}(s)|.$
	
	By Markov's inequality and the independence of $\Xlimit_u$, we have
	$$\Pmb \left(\frac{1}{n} \sum_{i=1}^n V_i^n > \varepsilon\right) \le \exp \left( -n \sup_{\theta \ge 0} \left[ \varepsilon\theta - \frac{1}{n} \sum_{i=1}^n \log \Emb \exp(\theta V_i^n) \right] \right).$$
	But for any given $\theta \ge 0$, with $C(T)$ and $\theta_T$ as in Lemma \ref{lem:regularity-moment}, we have
	$$\Emb \exp(\theta V_i^n) \le \Emb \exp\left(\frac{\theta^2}{4\theta_T} + \theta_T(V_i^n)^2 \right) \le (1+C(T)\Delta)\exp\left(\frac{\theta^2}{4\theta_T}\right).$$
	Therefore
	\begin{align*}
		\Pmb \left(\frac{1}{n} \sum_{i=1}^n V_i^n > \varepsilon\right) & \le \exp \left( -n \sup_{\theta \ge 0} \left[ \varepsilon\theta - \frac{\theta^2}{4\theta_T} - \log(1+C(T)\Delta) \right] \right) \\
		& = \exp \left( -n \left[ \theta_T\varepsilon^2 - \log(1+C(T)\Delta) \right] \right) \le \exp \left( -n \left[ \theta_T\varepsilon^2 - C(T)\Delta \right] \right).
	\end{align*}
	This completes the proof.
\end{proof}

\section{Sub-Gaussian inequality}
\label{sec:appendix-2}

\begin{proof}[Proof of Lemma \ref{lem:sub-Gaussian}]
	Since $\Emb Y=0$, we have
	\begin{equation*}
		\Emb[e^{\lambda \cdot Y}]
		= 1 + \Emb\left[ \sum_{k=2}^\infty \frac{(\lambda \cdot Y)^k}{k!} \right] 
		\le 1 + \frac{|\lambda|^2}{2} \Emb\left[ \sum_{k=2}^\infty \frac{|\lambda|^{k-2} |Y|^k}{(k-2)!} \right] 
		= 1 + \frac{|\lambda|^2}{2} \Emb\left[ |Y|^2 e^{|\lambda||Y|} \right].
	\end{equation*}
	Since $x^2 < 2a(1+\frac{x^2}{2a}) \le 2ae^{x^2/2a}$ and $tx \le \frac{at^2}{2} + \frac{x^2}{2a}$ for $x,t \in \Rmb$, we have
	\begin{align*}
		\Emb[e^{\lambda \cdot Y}] & \le 1 + \frac{|\lambda|^2}{2} \Emb\left[ 2a \exp \left(\frac{|Y|^2}{2a}\right) \exp \left(\frac{a|\lambda|^2}{2} + \frac{|Y|^2}{2a}\right) \right] \\
		& = 1 + a|\lambda|^2 \exp \left(\frac{a|\lambda|^2}{2}\right) \Emb\left[ \exp \left(\frac{|Y|^2}{a}\right) \right] \\
		& \le \left(1 + 2a|\lambda|^2 \right) \exp \left(\frac{a|\lambda|^2}{2}\right) \\
		& \le \exp \left(\frac{5a}{2}|\lambda|^2\right),
	\end{align*}
	where the last line follows on noting that $1+2ax \le e^{2ax}$ for $x \in \Rmb$.
\end{proof}

\section{Bounds of independent random variables}
\label{sec:appendix-3}

In this section we prove the following heterogeneous version of the result with i.i.d.\ setup in \cite[Theorem 2.1]{BolleyGuillinVillani2007quantitative} .
We will keep the use of $X$, $\mu$, etc.\ as in \cite{BolleyGuillinVillani2007quantitative}.
These notations are used in this section only.

\begin{Proposition}
	\label{prop:concentration-empirical}
	Let $\{X^i:i \in \Nmb\}$ be independent $\Rd$-valued random variables. Write
	\begin{equation*}
		\muhat^n := \frac{1}{n} \sum_{i=1}^n \delta_{X^i}, \quad \mu_i := \Lmc(X^i), \quad \mu^n := \frac{1}{n} \sum_{i=1}^n \mu_i.
	\end{equation*}
	Suppose $E_\alpha := \sup_{i \in \Nmb} \int_{\Rd} e^{\alpha |x|^2} \, \mu_i(dx) < \infty$ for some $\alpha > 0$.
	Let $p \in [1,2]$.
	Suppose there is some $\lambda>0$ such that $\mu^n$ satisfies the Talagrand $T_p(\lambda)$ inequality for each $n \in \Nmb$, namely
	\begin{equation}
		\label{eq:Talagrand}
		W_p(\nu,\mu^n) \le \sqrt{\frac{2}{\lambda} H(\nu\|\mu^n)}
	\end{equation}
	for any $\nu \in \Pmc(\Rd)$, where $H(\cdot\|\cdot)$ is the relative entropy.
	Then, for any $d'>d$ and $\lambda'<\lambda$, there exists some $N_0 \in \Nmb$, depending only on $\lambda',d',\alpha,E_\alpha$, such that for any $\varepsilon>0$ and $n \ge N_0 \max(\varepsilon^{-(d'+2)},1)$,
	\begin{equation*}
		\Pmb\left( W_p(\muhat^n,\mu^n) > \varepsilon \right) \le \exp \left(-\gamma_p\frac{\lambda'}{2} n\varepsilon^2 \right),
	\end{equation*}
	where
	\begin{equation*}
		\gamma_p = \begin{cases}
		1, & \mbox{if } 1 \le p < 2, \\
		3-2\sqrt{2}, & \mbox{if } p=2.
		\end{cases}
	\end{equation*}
\end{Proposition}

\begin{Remark}
	The assumption that $\mu^n$ satisfies the $T_p(\lambda)$ inequality, for some $p \in [1,2]$ and $\lambda>0$, implies the existence of a square-exponential moment of $\mu^n$.
	However, this is different from the assumption $E_\alpha < \infty$ which is a uniform bound on the square-exponential moment of each $\mu_i$.
	The latter is not redundant and will be used in the proof of Proposition \ref{prop:concentration-empirical}. 
\end{Remark}

\begin{proof}[Proof of Proposition \ref{prop:concentration-empirical}]
The proof follows the three-step argument in \cite[Section 3.1]{BolleyGuillinVillani2007quantitative} and hence we only report what is different.
In particular, details are provided for Step 1 and a part of Step 2 where the heterogeneity requires some additional estimates.
Only a sketch is provided for the remaining part of Step 2 and Step 3.

\textbf{Step 1 (Truncation)} Let $R>0$ and $B_R:= \{x \in \Rd : |x| \le R\}$.
Truncate $\mu_i$ into a probability measure $\mu_{i,R}$ on the ball $B_R$, namely 
$\mu_{i,R}(\cdot) := \one_{B_R} \mu_i(\cdot)/\mu_i(B_R)$. 
Let $\mu^n_R := \frac{1}{n} \sum_{i=1}^n \mu_{i,R}$.
 
Let $\{Y^i:i \in \Nmb\}$ be independent $\Rd$-valued random variables such that $\Lmc(Y^i)=\mu_{i,R}$, and $\{X^i:i \in \Nmb\}$ and $\{Y^i:i \in \Nmb\}$ are independent.
Define
\begin{equation*}
	X_R^i := \begin{cases}
	X^i, \: \mbox{ if } |X^i| \le R, \\
	Y^i, \: \mbox{ if } |X^i| > R.
	\end{cases}
\end{equation*}
Then
\begin{equation}
	\label{eq:app-1}
	|X^i-X^i_R| = |X^i-Y^i| \one_{|X^i| > R} \le (|X^i|+R) \one_{|X^i| > R} \le 2|X^i| \one_{|X^i| > R}
\end{equation}
and hence
\begin{equation*}
	W_p^p(\mu^n, \mu^n_R) \le \frac{1}{n} \sum_{i=1}^n \Emb |X^i-X^i_R|^p \le \frac{1}{n} \sum_{i=1}^n 2^p \Emb \left( |X^i|^p \one_{|X^i| > R} \right) = \frac{1}{n} \sum_{i=1}^n 2^p \int_{\{|x| > R\}} |x|^p \,\mu_i(dx).
\end{equation*}
If $R \ge \sqrt{p/(2\alpha)}$, then the function $r \mapsto r^p/e^{\alpha r^2}$ is nonincreasing for $r \ge R$, and then
\begin{equation*}
	W_p^p(\mu^n, \mu^n_R) \le \frac{1}{n} \sum_{i=1}^n 2^p \frac{R^p}{e^{\alpha R^2}} \int_{\{|x| > R\}} e^{\alpha |x|^2} \,\mu_i(dx).
\end{equation*}
We conclude that 
\begin{equation}
	\label{eq:21}
	W_p^p(\mu^n, \mu^n_R) \le 2^p E_\alpha R^p e^{-\alpha R^2} \quad \mbox{ for all } R \ge \sqrt{p/(2\alpha)}.
\end{equation}

On the other hand, the empirical measures
\begin{equation*}
	\muhat^n = \frac{1}{n} \sum_{i=1}^n \delta_{X^i}, \quad \muhat^n_R := \frac{1}{n} \sum_{i=1}^n \delta_{X^i_R}
\end{equation*}
satisfy
\begin{equation*}
	W_p^p(\muhat^n_R, \muhat^n) \le \frac{1}{n} \sum_{i=1}^n |X_R^i-X^i|^p \le \frac{1}{n} \sum_{i=1}^n Z^i
\end{equation*}
by \eqref{eq:app-1}, where $Z^i := 2^p |X^i|^p \one_{|X^i| > R}$, $i=1,\dotsc,n$.
Then, for any $p \in [1,2]$, we can introduce parameters $\varepsilon,\theta>0$, and use Chebyshev's exponential inequality and the independence of $Z^i$ to obtain
\begin{equation}
	\Pmb\left[ W_p(\muhat^n_R, \muhat^n) > \varepsilon \right] \le \Pmb\left[ \frac{1}{n} \sum_{i=1}^n Z^i > \varepsilon^p \right] \le \exp \left( -n\theta \varepsilon^p + \sum_{i=1}^n \log \Emb \exp( \theta Z^i) \right). \label{eq:22}
\end{equation}

In the case when $p \in[1,2)$, for any $\alpha_1 \in (0,\alpha)$, there exists some constant $R_1=R_1(\alpha_1,p)$ such that
	$2^{p} \theta r^{p} \leq \alpha_{1} r^{2}$
for all $\theta > 0$ and $r \ge R_1 \theta^{\frac{1}{2-p}}$.
Hence
\begin{align*}
	\Emb \exp \left(\theta Z^i \right) & \le \Emb \exp \left(\alpha_1|X^i|^2 \one_{|X^i|>R}\right) \le 1 + \Emb \left( e^{\alpha_1|X^i|^2} \one_{|X^i|>R} \right) \le 1 + e^{(\alpha_1-\alpha) R^2} E_\alpha
\end{align*}
for $R \ge R_1 \theta^{\frac{1}{2-p}}$.
From this and \eqref{eq:22} we have
\begin{equation*}
	\Pmb\left[ W_p(\muhat^n_R, \muhat^n) > \varepsilon \right] \le \exp \left( -n\left[ \theta \varepsilon^p - E_\alpha e^{(\alpha_1-\alpha) R^2} \right] \right).
\end{equation*}
From this, \eqref{eq:21} and the triangle inequality for $W_p$, we have
\begin{align}
	& \Pmb\left[ W_p(\mu^n, \muhat^n) > \varepsilon \right] \notag \\
	& \le \Pmb\left[ W_p(\mu^n, \mu^n_R) + W_p(\mu^n_R, \muhat^n_R) + W_p(\muhat^n_R, \muhat^n) > \varepsilon \right] \notag \\
	& \le \Pmb\left[ W_p(\mu^n_R, \muhat^n_R) > \eta \varepsilon - 2 E_\alpha^{1/p} R e^{-\frac{\alpha}{p} R^2} \right] + \Pmb\left[ W_p(\muhat^n_R, \muhat^n) > (1-\eta)\varepsilon \right] \notag \\
	& \le \Pmb\left[ W_p(\mu^n_R, \muhat^n_R) > \eta \varepsilon - 2 E_\alpha^{1/p} R e^{-\frac{\alpha}{p} R^2} \right] + \exp \left( -n\left[ \theta (1-\eta)^p\varepsilon^p - E_\alpha e^{(\alpha_1-\alpha) R^2} \right] \right). \label{eq:24}
\end{align}
This estimate was established for any given $p \in [1,2)$, $\eta \in (0,1)$, $\varepsilon, \theta > 0$, $\alpha_1 \in (0,\alpha)$ and $R \ge \sqrt{p/(2\alpha)} \vee R_1 \theta^{\frac{1}{2-p}}$, where $R_1$ is a constant depending only on $\alpha_1$ and $p$.

In the case when $p=2$, we let $Z^i := |X^i-Y^i|^2 \one_{|X^i| > R}$, $i=1,\dotsc,n$.
Starting from \eqref{eq:22} again, we choose $\alpha_1 \in (0,\alpha)$ and then $\theta:=\alpha_1/2$.
By definition of $Z^i$ and $\mu_{i,R}$,
\begin{align*}
	\Emb \left[ \exp \left( \frac{\alpha_1}{2} Z^i \right) \right] & = \int_{\Rmb^{2d}} \exp \left( \frac{\alpha_1}{2} |y-x|^2 \one_{|x| \ge R} \right) \mu_i(dx)\,\mu_{i,R}(dy) \\
	& = \mu_i(B_R) + \frac{1}{\mu_i(B_R)} \int_{|y| \le R} \int_{|x| \ge R} \exp \left( \frac{\alpha_1}{2} |y-x|^2 \right) \mu_i(dx)\,\mu_i(dy) \\
	& \le 1 + \left( 1 - E_\alpha e^{-\alpha R^2} \right)^{-1} \int_{|y| \le R} e^{\alpha_1 |y|^2} \,\mu_i(dy) \int_{|x| \ge R} e^{\alpha_1 |x|^2} \,\mu_i(dx) \\
	& \le 1 + 2 E_\alpha^2 e^{(\alpha_1-\alpha)R^2}
\end{align*}
for $R \ge R_2$, where $R_2=R_2(\alpha,E_\alpha)$ is some constant.
From this and \eqref{eq:22} we have
\begin{equation*}
	\Pmb \left[W_2(\muhat^n_R, \muhat^n) > \varepsilon \right] \le \exp \left( -n \left[ \frac{\alpha_1}{2} \varepsilon^2 - 2E_\alpha^2 e^{(\alpha_1-\alpha)R^2} \right] \right).
\end{equation*}
So a similar argument as in \eqref{eq:24} gives
\begin{align*}
	\Pmb\left[ W_2(\mu^n, \muhat^n) > \varepsilon \right] & \le \Pmb\left[ W_2(\mu^n_R, \muhat^n_R) > \eta \varepsilon - 2 E_\alpha^{1/2} R e^{-\frac{\alpha}{2} R^2} \right] 
	\\
	& \qquad + \exp \left( -n\left[ \frac{\alpha_1}{2} (1-\eta)^2\varepsilon^2 - 2E_\alpha^2 e^{(\alpha_1-\alpha) R^2} \right] \right). 
\end{align*}
This estimate was established for any given $\eta \in (0,1)$, $\varepsilon > 0$, $\alpha_1 \in (0,\alpha)$ and $R \ge \sqrt{p/(2\alpha)} \vee R_2$, where $R_2$ is a constant depending only on $\alpha$ and $E_\alpha$.

So, apart from some error terms, it will be sufficient to establish the result for the ``truncated" law $\mu^n_R$, whose support lies in the compact set $B_R$.

Next we prove that $\mu^n_R$ satisfies some modified $T_p$ inequality.
Let $\nu$ be a probability measure on $B_R$, absolutely continuous with respect to $\mu$ (and hence with respect to $\mu^n_R$).
Then, when $R \ge R_3$ for some constant $R_3=R_3(E_\alpha)$, we can write
\begin{align*}
	H(\nu\|\mu^n_R) - H(\nu\|\mu^n) & = \int_{B_R} \log \frac{d\nu}{d\mu^n_R} \,d\nu - \int_{B_R} \log \frac{d\nu}{d\mu^n} \,d\nu \\
	& = \log [\mu^n(B_R)] \ge \log \left( 1- e^{-\alpha R^2} \int_{\Rmb^d} e^{\alpha |x|^2} \,\mu^n(dx) \right) \\
	& \ge \log \left( 1- E_\alpha e^{-\alpha R^2} \right) \ge -2 E_\alpha e^{-\alpha R^2}.
\end{align*}
Since $\mu^n$ satisfies the $T_p(\lambda)$ inequality,
$$H(\nu\|\mu^n) \ge \frac{\lambda}{2} W_p^2(\mu^n,\nu) \ge \frac{\lambda}{2} \left( W_p(\mu^n_R,\nu) - W_p(\mu^n_R,\mu^n) \right)^2$$
by triangle inequality.
Combining these two displays, we obtain
$$H(\nu\|\mu^n_R) \ge \frac{\lambda}{2} \left( W_p(\mu^n_R,\nu) - W_p(\mu^n_R,\mu^n) \right)^2 - 2 E_\alpha e^{-\alpha R^2}.$$
From this, \eqref{eq:21} and elementary inequality
\begin{equation*}
	\forall a\in(0,1),\: \exists\, C_a>0 \mbox{ such that } \forall x,y\in\Rmb,\: (x-y)^2 \ge (1-a)x^2-C_ay^2,
\end{equation*}
We deduce that for any $\lambda_1<\lambda$ there exists some constant $K$ such that
\begin{equation*}
	H(\nu\|\mu^n_R) \ge \frac{\lambda_1}{2} W_p^2(\mu^n_R,\nu) - KR^2e^{-\alpha R^2}.
\end{equation*}

\textbf{Step 2 (Covering by small balls)} In this second step we derive quantitative estimates on $\muhat^n_R$.
Let $\phi$ be a bounded continuous function on $B_R$, and let $U$ be a Borel set in $\Pmc(B_R)$.
By Chebyshev's exponential inequality, the independence of the variables $X_R^i$, and the concavity of $\log(\cdot)$,
\begin{align*}
	\Pmb(\muhat_R^n \in U) & \le \exp\left( -n \inf_{\nu \in U} \int_{B_R} \phi \,d\nu \right) \Emb \left(e^{n\int_{B_R} \phi\,d\muhat_R^n}\right) \\
	& = \exp\left( -n \inf_{\nu \in U} \left[ \int_{B_R} \phi \,d\nu - \frac{1}{n} \log \Emb \left(e^{n\int_{B_R} \phi\,d\muhat_R^n}\right) \right] \right) \\
	& = \exp\left( -n \inf_{\nu \in U} \left[ \int_{B_R} \phi \,d\nu - \frac{1}{n} \sum_{i=1}^n \log \Emb \left(e^{\phi(X_R^i)}\right) \right] \right) \\
	& \le \exp\left( -n \inf_{\nu \in U} \left[ \int_{B_R} \phi \,d\nu - \log \frac{1}{n} \sum_{i=1}^n  \Emb \left(e^{\phi(X_R^i)}\right) \right] \right) \\
	& = \exp\left( -n \inf_{\nu \in U} \left[ \int_{B_R} \phi \,d\nu - \log \int_{B_R} e^\phi \, d\mu_R^n \right] \right).
\end{align*}
As $\phi$ is arbitrary, we can show that \cite[Equation (30)]{BolleyGuillinVillani2007quantitative} holds if $U$ is convex and compact, namely
\begin{equation*}
	\Pmb(\muhat_R^n \in U) \le \exp\left( -n \inf_{\nu \in U} H(\nu\|\mu_R^n) \right).
\end{equation*}
Following the argument in \cite[Pages 557--559]{BolleyGuillinVillani2007quantitative}, we can show that equations (36) and (37) therein hold. 
More precisely, given $p \in [1,2)$, $\lambda_2<\lambda$ and $\alpha_1<\alpha$, there exist some constants $K_1,K_2,K_3,R_4 \in (0,\infty)$ such that for all $\varepsilon, \zeta>0$ and $R \ge R_4 \max(1,\zeta^{\frac{1}{2-p}})$,
\begin{align}
	\Pmb\left(W_p(\muhat^n,\mu^n)>\varepsilon\right) & \le \left( K_2\frac{R}{\varepsilon} \vee 1 \right)^{K_2\left(\frac{R}{\varepsilon}\right)^d} \exp\left(-n \left[ \frac{\lambda_2}{2} \varepsilon^2 - K_1R^2e^{-\alpha R^2} \right] \right) \notag \\
	& + \exp \left( -n\left[ K_3 \zeta \varepsilon^p - K_4 e^{(\alpha_1-\alpha) R^2} \right] \right) 
	\label{eq:36} 
\end{align}
for some constant $K_4=K_4(\theta,\alpha_1)$;
while in the case $p=2$,
\begin{align}
	\Pmb\left(W_2(\muhat^n,\mu^n)>\varepsilon\right) & \le \left( K_2\frac{R}{\varepsilon} \vee 1 \right)^{K_2\left(\frac{R}{\varepsilon}\right)^d} \exp\left(-n \left[ \frac{\lambda_2}{2} \eta^2\varepsilon^2 - K_1R^2e^{-\alpha R^2} \right] \right) \notag \\
	& + \exp \left( -n\left[ \frac{\alpha_1}{2}(1-\eta)^2 \varepsilon^2 - K_4 e^{(\alpha_1-\alpha) R^2} \right] \right) 
	\label{eq:37} 
\end{align}
for any $\eta \in (0,1)$ and $R \ge R_4$.

\textbf{Step 3 (Choice of the parameters)} 
The argument is the same as that in \cite[Pages 559--561]{BolleyGuillinVillani2007quantitative}. 
Therefore we only provide a sketch here.
First consider the case $p \in [1,2)$.
Let $\lambda'<\lambda_2$, $\alpha'<\alpha$ and $d_1>d$.
Using \eqref{eq:36} one can show that \cite[(38)]{BolleyGuillinVillani2007quantitative} holds, that is,
\begin{equation}
	\label{eq:step3}
	\Pmb(W_p(\muhat^n,\mu^n)>\varepsilon) \le \exp\left( -\frac{\lambda'}{2}n\varepsilon^2 \right) + \exp(-\alpha'n\varepsilon^2)
\end{equation}
as soon as
$$R^2 \ge R_2\max\left(1, \varepsilon^2, \log(\frac{1}{\varepsilon^2})\right), \quad n\varepsilon^{d_1+2} \ge K_5R^{d_1}$$
for some constants $R_2$ and $K_5$ depending on $\mu^n$ only through $\lambda$, $\alpha$ and $E_\alpha$.
By choosing parameters $R, n, d'$ carefully as in \cite[(39)]{BolleyGuillinVillani2007quantitative}, one can find some $N_0 = N_0(\lambda',d',\alpha,E_\alpha)$ such that for all $\varepsilon > 0$, \eqref{eq:step3} holds
as soon as $n \ge N_0 \max(\varepsilon^{-(d'+2)},1)$.
Choosing $\alpha < \lambda/2$ gives the desired upper bound for the right hand side of \eqref{eq:step3}.

Finally consider the case $p=2$. 
Given $\lambda_3<\lambda_2$ and $\alpha_2<\alpha_1$, using \eqref{eq:37} one can show that similar estimates as in \eqref{eq:step3} holds, that is,
\begin{equation*}
	\Pmb(W_2(\muhat^n,\mu^n)>\varepsilon) \le \exp\left( -\frac{\lambda_3}{2}\eta^2n\varepsilon^2 \right) + \exp\left(-\frac{\alpha_2}{2}(1-\eta)^2n\varepsilon^2\right).
\end{equation*}
Taking $\alpha_2=\lambda_3/2$ and $\eta=\sqrt{2}-1$ gives the desired result.
\end{proof}

\bibliographystyle{plain}

\begin{bibdiv}
\begin{biblist}

\bib{BarreDobsonOttobreZatorska2021fast}{article}{
      author={Barr{\'e}, Julien},
      author={Dobson, Paul},
      author={Ottobre, Michela},
      author={Zatorska, Ewelina},
       title={Fast non-mean-field networks: Uniform in time averaging},
        date={2021},
     journal={SIAM Journal on Mathematical Analysis},
      volume={53},
      number={1},
       pages={937\ndash 972},
}

\bib{BayraktarChakrabortyWu2020graphon}{article}{
      author={Bayraktar, Erhan},
      author={Chakraborty, Suman},
      author={Wu, Ruoyu},
       title={Graphon mean field systems},
        date={2022},
     journal={The Annals of Applied Probability, accepted},
}

\bib{BayraktarWu2019mean}{article}{
      author={Bayraktar, Erhan},
      author={Wu, Ruoyu},
       title={Mean field interaction on random graphs with dynamically changing
  multi-color edges},
        date={2021},
     journal={Stochastic Processes and their Applications},
      volume={141},
       pages={197\ndash 244},
}

\bib{BayraktarWu2020stationarity}{article}{
      author={Bayraktar, Erhan},
      author={Wu, Ruoyu},
       title={Stationarity and uniform in time convergence for the graphon
  particle system},
        date={2022},
     journal={Stochastic Processes and their Applications},
      volume={150},
       pages={532\ndash 568},
}

\bib{BayraktarWuZhang2022propagation}{article}{
      author={Bayraktar, Erhan},
      author={Wu, Ruoyu},
      author={Zhang, Xin},
       title={Propagation of chaos of forward-backward stochastic differential
  equations with graphon interactions},
        date={2022},
     journal={arXiv preprint arXiv:2202.08163},
}

\bib{BetCoppiniNardi2020weakly}{article}{
      author={Bet, Gianmarco},
      author={Coppini, Fabio},
      author={Nardi, Francesca~R},
       title={Weakly interacting oscillators on dense random graphs},
        date={2020},
     journal={arXiv preprint arXiv:2006.07670},
}

\bib{BhamidiBudhirajaWu2019weakly}{article}{
      author={Bhamidi, Shankar},
      author={Budhiraja, Amarjit},
      author={Wu, Ruoyu},
       title={Weakly interacting particle systems on inhomogeneous random
  graphs},
        date={2019},
     journal={Stochastic Processes and their Applications},
      volume={129},
      number={6},
       pages={2174\ndash 2206},
}

\bib{BolleyGuillinMalrieu2010trend}{article}{
      author={Bolley, Fran{\c{c}}ois},
      author={Guillin, Arnaud},
      author={Malrieu, Florent},
       title={Trend to equilibrium and particle approximation for a weakly
  selfconsistent vlasov-fokker-planck equation},
        date={2010},
     journal={ESAIM: Mathematical Modelling and Numerical Analysis},
      volume={44},
      number={5},
       pages={867\ndash 884},
}

\bib{BolleyGuillinVillani2007quantitative}{article}{
      author={Bolley, Fran{\c{c}}ois},
      author={Guillin, Arnaud},
      author={Villani, C{\'e}dric},
       title={Quantitative concentration inequalities for empirical measures on
  non-compact spaces},
        date={2007},
     journal={Probability Theory and Related Fields},
      volume={137},
      number={3-4},
       pages={541\ndash 593},
}

\bib{BolleyVillani2005weighted}{inproceedings}{
      author={Bolley, Fran{\c{c}}ois},
      author={Villani, C{\'e}dric},
       title={Weighted csisz{\'a}r-kullback-pinsker inequalities and
  applications to transportation inequalities},
        date={2005},
   booktitle={Annales de la facult{\'e} des sciences de toulouse:
  Math{\'e}matiques},
      volume={14},
       pages={331\ndash 352},
}

\bib{BudhirajaFan2017}{article}{
      author={Budhiraja, Amarjit},
      author={Fan, Wai-Tong~Louis},
       title={Uniform in time interacting particle approximations for nonlinear
  equations of {P}atlak-{K}eller-{S}egel type},
        date={2017},
     journal={Electron. J. Probab.},
      volume={22},
       pages={37 pp.},
         url={https://doi.org/10.1214/17-EJP25},
}

\bib{BudhirajaMukherjeeWu2019supermarket}{article}{
      author={Budhiraja, Amarjit},
      author={Mukherjee, Debankur},
      author={Wu, Ruoyu},
       title={Supermarket model on graphs},
        date={201906},
     journal={Ann. Appl. Probab.},
      volume={29},
      number={3},
       pages={1740\ndash 1777},
         url={https://doi.org/10.1214/18-AAP1437},
}

\bib{CainesHuang2018graphon}{inproceedings}{
      author={Caines, Peter~E},
      author={Huang, Minyi},
       title={Graphon mean field games and the {GMFG} equations},
organization={IEEE},
        date={2018},
   booktitle={{2018 IEEE Conference on Decision and Control (CDC)}},
       pages={4129\ndash 4134},
}

\bib{CainesHuang2020graphon}{article}{
      author={Caines, Peter~E},
      author={Huang, Minyi},
       title={Graphon mean field games and their equations},
        date={2021},
     journal={SIAM Journal on Control and Optimization},
      volume={59},
      number={6},
       pages={4373\ndash 4399},
}

\bib{Carmona2019stochastic}{article}{
      author={Carmona, Ren{\'e}},
      author={Cooney, Daniel~B},
      author={Graves, Christy~V},
      author={Lauriere, Mathieu},
       title={{Stochastic graphon games: I. the static case}},
        date={2022},
     journal={Mathematics of Operations Research},
      volume={47},
      number={1},
       pages={750\ndash 778},
}

\bib{Coppini2021note}{article}{
      author={Coppini, Fabio},
       title={{A Note on Fokker--Planck Equations and Graphons}},
        date={2022},
     journal={Journal of Statistical Physics},
      volume={187},
      number={2},
       pages={1\ndash 12},
}

\bib{Coppini2019long}{article}{
      author={Coppini, Fabio},
       title={Long time dynamics for interacting oscillators on graphs},
        date={2022},
     journal={The Annals of Applied Probability},
      volume={32},
      number={1},
       pages={360\ndash 391},
}

\bib{CoppiniDietertGiacomin2019law}{article}{
      author={Coppini, Fabio},
      author={Dietert, Helge},
      author={Giacomin, Giambattista},
       title={A law of large numbers and large deviations for interacting
  diffusions on {E}rdős–{R}ényi graphs},
        date={2019},
     journal={Stochastics and Dynamics},
      volume={0},
      number={0},
       pages={2050010},
      eprint={https://doi.org/10.1142/S0219493720500100},
         url={https://doi.org/10.1142/S0219493720500100},
}

\bib{Delarue2017mean}{article}{
      author={Delarue, Fran{\c{c}}ois},
       title={{Mean field games: A toy model on an Erd{\"o}s-Renyi graph.}},
        date={2017},
     journal={ESAIM: Proceedings and Surveys},
      volume={60},
       pages={1\ndash 26},
}

\bib{DelarueLackerRamanan2020from}{article}{
      author={Delarue, François},
      author={Lacker, Daniel},
      author={Ramanan, Kavita},
       title={{From the master equation to mean field game limit theory: Large
  deviations and concentration of measure}},
        date={2020},
     journal={The Annals of Probability},
      volume={48},
      number={1},
       pages={211 \ndash  263},
         url={https://doi.org/10.1214/19-AOP1359},
}

\bib{Delattre2016}{article}{
      author={Delattre, Sylvain},
      author={Giacomin, Giambattista},
      author={Lu{\c{c}}on, Eric},
       title={A note on dynamical models on random graphs and
  {F}okker--{P}lanck equations},
        date={2016},
        ISSN={1572-9613},
     journal={Journal of Statistical Physics},
      volume={165},
      number={4},
       pages={785\ndash 798},
         url={https://doi.org/10.1007/s10955-016-1652-3},
}

\bib{DupuisMedvedev2020large}{article}{
      author={Dupuis, Paul},
      author={Medvedev, Georgi~S},
       title={The large deviation principle for interacting dynamical systems
  on random graphs},
        date={2022},
     journal={Communications in Mathematical Physics},
      volume={390},
      number={2},
       pages={545\ndash 575},
}

\bib{GaoTchuendomCaines2020linear}{article}{
      author={Gao, Shuang},
      author={Tchuendom, Rinel~Foguen},
      author={Caines, Peter~E},
       title={Linear quadratic graphon field games},
        date={2020},
     journal={arXiv preprint arXiv:2006.03964},
}

\bib{KaratzasShreve1991brownian}{book}{
      author={Karatzas, I.},
      author={Shreve, S.~E.},
       title={{Brownian Motion and Stochastic Calculus}},
      series={Graduate Texts in Mathematics},
   publisher={Springer New York},
        date={1991},
      volume={113},
        ISBN={9780387976556},
}

\bib{Kolokoltsov2010}{book}{
      author={Kolokoltsov, V.~N.},
       title={{Nonlinear Markov Processes and Kinetic Equations}},
      series={Cambridge Tracts in Mathematics},
   publisher={Cambridge University Press},
        date={2010},
      volume={182},
}

\bib{LackerSoret2020case}{article}{
      author={Lacker, Daniel},
      author={Soret, Agathe},
       title={A case study on stochastic games on large graphs in mean field
  and sparse regimes},
        date={2022},
     journal={Mathematics of Operations Research},
      volume={47},
      number={2},
       pages={1530\ndash 1565},
}

\bib{Lovasz2012large}{book}{
      author={Lov{\'a}sz, L{\'a}szl{\'o}},
       title={Large networks and graph limits},
   publisher={American Mathematical Soc.},
        date={2012},
      volume={60},
}

\bib{Lucon2020quenched}{article}{
      author={Lu{\c{c}}on, Eric},
       title={Quenched asymptotics for interacting diffusions on inhomogeneous
  random graphs},
        date={2020},
     journal={Stochastic Processes and their Applications},
}

\bib{McKean1967propagation}{incollection}{
      author={McKean, H.~P.},
       title={{Propagation of chaos for a class of non-linear parabolic
  equations}},
        date={1967},
   booktitle={Stochastic differential equations ({L}ecture {S}eries in
  {D}ifferential {E}quations, {S}ession 7, {C}atholic {U}niversity, 1967)},
   publisher={Air Force Office Sci. Res., Arlington, Va.},
       pages={41\ndash 57},
}

\bib{Medvedev2014nonlinear2}{article}{
      author={Medvedev, Georgi~S},
       title={The nonlinear heat equation on dense graphs and graph limits},
        date={2014},
     journal={SIAM Journal on Mathematical Analysis},
      volume={46},
      number={4},
       pages={2743\ndash 2766},
}

\bib{Medvedev2014nonlinear1}{article}{
      author={Medvedev, Georgi~S},
       title={The nonlinear heat equation on w-random graphs},
        date={2014},
     journal={Archive for Rational Mechanics and Analysis},
      volume={212},
      number={3},
       pages={781\ndash 803},
}

\bib{Medvedev2018continuum}{article}{
      author={Medvedev, Georgi~S},
       title={{The continuum limit of the Kuramoto model on sparse random
  graphs}},
        date={2018},
     journal={arXiv preprint arXiv:1802.03787},
}

\bib{OliveiraReis2019interacting}{article}{
      author={Oliveira, Roberto},
      author={Reis, Guilherme},
       title={Interacting diffusions on random graphs with diverging average
  degrees: Hydrodynamics and large deviations},
        date={2019},
     journal={Journal of Statistical Physics},
      volume={176},
       pages={1057\ndash 1087},
}

\bib{PariseOzdaglar2019graphon}{article}{
      author={Parise, Francesca},
      author={Ozdaglar, Asuman~E},
       title={Graphon games: A statistical framework for network games and
  interventions},
        date={2019},
     journal={Available at SSRN: https://ssrn.com/abstract=3437293},
}

\bib{Sznitman1991}{incollection}{
      author={Sznitman, A-S.},
       title={{Topics in propagation of chaos}},
        date={1991},
   booktitle={Ecole d'{E}t{\'e} de {P}robabilit{\'e}s de {S}aint-{F}lour
  {XIX}---1989},
      editor={Hennequin, Paul-Louis},
      series={Lecture Notes in Mathematics},
      volume={1464},
   publisher={Springer Berlin Heidelberg},
     address={Berlin, Heidelberg},
       pages={165\ndash 251},
}

\bib{VasalMishraVishwanath2020sequential}{inproceedings}{
      author={Vasal, Deepanshu},
      author={Mishra, Rajesh},
      author={Vishwanath, Sriram},
       title={Sequential decomposition of graphon mean field games},
organization={IEEE},
        date={2021},
   booktitle={{2021 American Control Conference (ACC)}},
       pages={730\ndash 736},
}

\bib{Veretennikov2006ergodic}{incollection}{
      author={Veretennikov, A~Yu},
       title={On ergodic measures for mckean-vlasov stochastic equations},
        date={2006},
   booktitle={Monte carlo and quasi-monte carlo methods 2004},
   publisher={Springer},
       pages={471\ndash 486},
}

\bib{Villani2008optimal}{book}{
      author={Villani, C{\'e}dric},
       title={Optimal transport: old and new},
   publisher={Springer Science \& Business Media},
        date={2008},
      volume={338},
}

\end{biblist}
\end{bibdiv}


\end{document}